\documentclass[a4paper,10pt]{article}
\usepackage{amssymb}
\usepackage{graphicx}
\usepackage{caption}
\usepackage{subcaption}
\usepackage{psfrag}
\usepackage{psfrag,picture}
\usepackage{stmaryrd}
\usepackage{amsmath}
\usepackage{pstricks,pst-3dplot}
\usepackage{booktabs}
\usepackage{bm}
\usepackage{amsthm}

\newtheorem{theorem}{Theorem}[section]
\newtheorem{lemma}[theorem]{Lemma}
\newcommand{\vertiii}[1]{{\left\vert\kern-0.25ex\left\vert\kern-0.25ex\left\vert #1 
    \right\vert\kern-0.25ex\right\vert\kern-0.25ex\right\vert}}
\newenvironment{remark}[1][Remark]{\begin{trivlist}
\item[\hskip \labelsep {\bfseries #1}]}{\end{trivlist}}

\title{Convergence of summation--by--parts finite difference methods for the wave equation}

\author{Siyang Wang\footnotemark[1]\ \footnotemark[2]
\and Gunilla Kreiss\footnotemark[2]}
\date{}
\begin{document}
\maketitle

\renewcommand{\thefootnote}{\fnsymbol{footnote}} 
\footnotetext[1]{Corresponding author, email: {siyang.wang@it.uu.se}}
 \footnotetext[2]{Department of Information Technology, Uppsala University, Uppsala, SE-751 05, Sweden}
\renewcommand{\thefootnote}{\arabic{footnote}}

\begin{abstract}
In this paper, we consider finite difference approximations of the second order wave equation.  We use finite difference operators satisfying the summation--by--parts property to discretize the equation in space. Boundary conditions and grid interface conditions are imposed by the simultaneous--approximation--term technique. Typically, the truncation error is larger at the grid points near a boundary or grid interface than that in the interior. Normal mode analysis can be used to analyze how the large truncation error affects the convergence rate of the underlying stable numerical scheme. If the semi--discretized equation satisfies a determinant condition, two orders are gained from the large truncation error. However, many interesting second order equations do not satisfy the determinant condition. We then carefully analyze the solution of the boundary system to derive a sharp estimate for the error in the solution and acquire the gain in convergence rate. The result shows that stability does not automatically yield a gain of two orders in convergence rate. The accuracy analysis is verified by numerical experiments. 
\end{abstract}



\pagestyle{myheadings}
\thispagestyle{plain}

\section{Introduction}
In many physical problems, such as acoustics, seismology, and electromagnetism, the underlying equations can be formulated as systems of second order time dependent hyperbolic partial differential equations. For wave propagation problems, it has been shown that high order accurate time marching methods as well as high order spatial discretization are more efficient to solve these problems on smooth domains \cite{Gustafsson2008,Kreiss1972}. The major difficulties with high order spatial discretization are the numerical treatment of boundary conditions, grid interface conditions and interface conditions at material interfaces.

High order finite difference methods have been widely used for wave propagation problems. We use summation-by-parts (SBP) finite difference operators \cite{Mattsson2013,Mattsson2004,Svard2014} to approximate spatial derivatives. In order to guarantee strict stability and maintain high accuracy, boundary conditions and grid interface conditions are imposed weakly by the simultaneous--approximation--term (SAT) technique \cite{Appelo2007,Carpenter1994,Mattsson2008,Mattsson2009}. This is often referred to as the SBP--SAT methodology, and a summary of it is made in \cite{Fernandez2014}. The energy method is used to determine the strength of the boundary enforcement and to derive an energy estimate for stability. Another candidate for the approximation of spatial derivatives is discontinuous Galerkin method, and it has been successfully used in \cite{Grote2006} for the wave equation. 

In the SBP--SAT framework, accuracy analysis has drawn less attention than stability. A commonly used measure for accuracy is convergence rate, typically in L$_2$ norm. The convergence rate indicates how fast the error in the numerical solution approaches zero as the mesh size $h$ goes to zero. The error in the numerical solution is caused by truncation error. The truncation error near a boundary is often larger than that in the interior of the computational domain, but with the large boundary truncation error localized at a fixed number of grid points that is independent of mesh refinement. As a consequence, its effect to the convergence rate may be weakened. It is well--known that by applying directly the energy method to the error equation, $1/2$ order is gained in convergence. However, this convergence result is in many cases suboptimal. When using a finite difference method with mesh size $h$ to solve a first order hyperbolic partial differential equation, it is in many cases possible to show that $(p+1)^{th}$ order convergence rate can be obtained if the boundary truncation error is $\mathcal{O}(h^p)$ \cite{Gustafsson1975,Gustafsson2008}. In other words, we gain one order in convergence rate. The technique used to analyze the effect of the boundary truncation error in \cite{Gustafsson1975} is normal mode analysis. The idea is based on Laplace transforming the error equation in time. 
The error equation due to boundary truncation error can be formulated as a system of linear equations, which is referred to as the boundary system. A gain of one order in convergence rate follows if the boundary system is nonsingular for all Re$(s)\geq 0$, where $s$ is the dual variable of time in Laplace space. For such cases we use the same terminology as in \cite{Gustafsson2013} and say that the boundary system satisfies the determinant condition. Also note that if a problem satisfies an energy estimate, the boundary system is nonsingular for Re$(s)>0$. However, an energy estimate does in general not imply that the determinant condition is satisfied.

Also for many other systems the boundary truncation errors have less severe effect than predicted by energy estimates. In \cite{Abarbanel2000} a parabolic problem is considered. Numerical experiments show a convergence rate $\min(2p,p+2)$, where $p$ is the order of the boundary truncation error, while a careful deviation of the energy estimate yields $\min(2p,p+3/2)$ as the convergence rate. This reaffirms that the energy estimate is an upper bound of the error, and indicates that the energy estimate is not always sharp. 

A more general result on convergence rate is presented in \cite{Svard2006}, where it is shown that $k$ orders can be recovered from boundary truncation error for a partial differential equation with $k^{th}$ spatial derivatives. Parabolic, incomplete parabolic and second order hyperbolic partial differential equations are considered, and for these equations we call $\min(2p,p+2)$ the optimal convergence rate. The condition for the optimal convergence is formulated in terms of a new stability concept, pointwise stability, but the underlying analysis of the optimal convergence is based on the determinant condition, and therefore the approach cannot be applied when the boundary system is singular at the origin.

When using SBP--SAT finite difference method to solve second order hyperbolic partial differential equations, there are many cases that the boundary system does not satisfy the determinant condition. In this paper, we aim at filling the gap in the study of convergence rate for such problems. We use the wave equation in second order form as our model problem. In particular, we consider problems with Dirichlet boundary condition, Neumann boundary condition and grid interface condition. We refer to them as the Dirichlet problem, the Neumann problem and the interface problem, respectively. By analyzing the error equation, we obtain a boundary system of linear equations. If the determinant condition is satisfied, we would get the optimal convergence rate as shown in \cite[Ch.12]{Gustafsson2013}. If the determinant condition is not satisfied on the imaginary axis, \cite{Gustafsson2013, Svard2006} are not applicable and one has to carefully derive an estimate for the solution of the boundary system to see how much is gained in convergence rate. 

In this paper, we find that 1) an energy estimate for the second order wave equation does not imply an optimal convergence rate; 2) the determinant condition is not necessary for an optimal convergence rate; 3) if there is an energy estimate but the determinant condition is not satisfied, there can be an optimal gain of order 2, or a non--optimal gain of order 1, or only a 1/2 order gain that is given by the energy estimate. We demonstrate how to analyze the boundary system to determine the accuracy gain. 

More specifically, for the Dirichlet problem, the boundary system is singular or nonsingular depending on the value of the penalty parameter in SAT. In this case, the optimal convergence rate is only obtained when the determinant condition is satisfied. 
For the Neumann and the interface problems, the determinant condition is not satisfied even though the schemes are stable. We will analyze the boundary system, and show how much is gained in convergence in those cases.
The analysis agrees well with numerical experiments. In \cite{Nissen2012,Nissen2013}, the same technique is also used to prove optimal convergence rate for Schr\"{o}dinger equation.

The structure of this paper is as follows. We start in \S\ref{1dhalfline} with the SBP--SAT method applied to the one dimensional wave equation with Dirichlet boundary condition. We apply the energy method and normal mode analysis, and derive results that correspond to the second, fourth and sixth order schemes. We then discuss the Neumann problem in \S\ref{section_Neumann}. In \S\ref{section_interface}, we consider one dimensional wave equation on a grid with a grid interface. We extend the analysis to two dimensions in \S\ref{2danalysis}. In \S\ref{numericalexperiments}, we perform numerical experiments. The computational convergence results support the accuracy analysis. The conclusions are drawn in \S\ref{conclusion}.

\section{One dimensional wave equation with Dirichlet boundary condition}\label{1dhalfline}
To describe the properties of the equation and the numerical methods, we need the following definitions. Let $w_1(x)$ and $w_2(x)$ be real--valued functions in L$_2[a_1,a_2]$. The inner product is defined as $(w_1,w_2)=\int_{a_1}^{a_2} w_1w_2 dx$. The corresponding norm is $\|w_1\|^2=(w_1,w_1)$.

The second order wave equation in one dimension takes the form 
\begin{equation}\label{1dwave1}
\begin{split}
&U_{tt} = U_{xx} + F,\ x\in [0,\infty),\ t\geq 0,\\
&U(x,0)=f^1(x),\ U_t(x,0)=f^2(x). 
\end{split}
\end{equation}
where $F$ is a forcing term.
We consider the half line problem with Dirichlet boundary condition:
\begin{equation}\label{1dwave2}
U(0,t)=g(t).
\end{equation}
The forcing function $F$, the initial data $f^1$, $f^2$ and the boundary data $g$ are compatible smooth functions with compact support. For the problem (\ref{1dwave1}) in a bounded domain, there is one boundary condition at each boundary. For the half line problem in consideration, the right boundary condition is substituted by requiring that the L$_2$ norm of the solution is bounded, i.e. $\|U(\cdot,t)\| < \infty$. The problem (\ref{1dwave1})--(\ref{1dwave2}) is well--posed if there is an energy estimate with $g(t)=0$ and the problem is boundary stable. In \cite[Definition 2.3]{Kreiss2012}, it is defined that the problem ($\ref{1dwave1}$)--($\ref{1dwave2}$) is boundary stable if with $f^1=f^2=0$ and $F=0$ there are constants $\eta_0\geq 0$, $\bar K>0$ and $\alpha>0$ independent of $g$ such that for all $\bar\eta>\eta_0$, $T\geq 0$,
\begin{equation*}
\int_0^T e^{-2\bar\eta t} |u(0,t)|^2 dt \leq \frac{\bar K}{\bar\eta^\alpha} \int_0^T e^{-2\bar\eta t} |g(t)|^2 dt.
\end{equation*}
It is obvious that ($\ref{1dwave1}$) is boundary stable with Dirichlet boundary condition.
To derive an energy estimate, we multiply Equation (\ref{1dwave1}) by $U_t$ and integrate by parts. With homogeneous Dirichlet boundary condition $g(t)=0$, we obtain
\begin{equation*}
\frac{d}{dt}\sqrt{E}\leq \|F\|,
\end{equation*}
where $E=\|U_t\|^2+\|U_x\|^2$ is the continuous energy. The energy estimate follows from Gronwall's lemma
\begin{equation}\label{1denergy}
\sqrt{E}\leq\sqrt{\|f^1_x\|^2+\|f^2\|^2}+\int_0^t\|F(\cdot,z)\|dz.
\end{equation}
Therefore, problem (\ref{1dwave1})--(\ref{1dwave2}) is well--posed. 

Next, we introduce the equidistant grid $x_i=ih,\ i=0,1,2,\cdots,$ and a grid function $u_i(t)\approx U(x_i,t)$. Furthermore, let $u(t)=[u_0(t),u_1(t),\cdots]^T$. We also define an inner product and norm for the grid functions $a$ and $b$ in $\mathcal{R}$ as $(a,b)_H=a^THb$ and $\|a\|_H^2=a^THa$, respectively, where $a^T$ denotes the transpose of $a$ and $H$ is a positive definite operator in the space of grid functions. An SBP operator approximates a derivative, and mimics the property of the continuous integration-by-parts via the inner product and norm defined above. 

In this paper, we use the diagonal norm SBP operator $D$ approximating second order derivative, i.e. $D\approx\frac{\partial^2}{\partial x^2}$. The order of accuracy is denoted by $2p,\ p=1,2,3,4,5$. For a $2p^{th}$ order diagonal norm SBP operator \cite{Mattsson2013,Mattsson2004}, the error for the approximation of the spatial derivative is $\mathcal{O}(h^{2p})$ in the interior, and $\mathcal{O}(h^{p})$ near the boundary. The SBP operator can be written as $D=H^{-1}(-A+BS)$, where $H$ is diagonal and positive definite, $A$ is symmetric positive semi-definite and $B=diag(-1,0,0,\cdots)$. $S$ is a one sided approximation of the first derivative at the boundary with the order of accuracy $p+1$. The elements in $D$ are proportional to $1/{h^2}$ while the elements in $H^{-1}$ and $S$ are proportional to $1/h$. Another version of the SBP operators are constructed with block norms, see \cite{Mattsson2013,Mattsson2004}. 

We use the SAT technique to impose the Dirichlet boundary condition weakly. The semi-discretized equation corresponding to (\ref{1dwave1}) and (\ref{1dwave2}) with homogeneous boundary data reads:
\begin{equation}\label{semidiscrete1dhomo}
u_{tt}=Du-H^{-1}S^TE_0 u-\frac{\tau}{h}H^{-1}E_0 u+F_g,
\end{equation}
where $e_0=[1,0,0,\cdots]^T$, $E_0=e_0e_0^T$ and $F_g$ is the grid function corresponding to $F(x,t)$. On the right hand side, the first term is an approximation of $U_{xx}$, while the second and third terms impose weakly the boundary condition $U(0,t)=0$. They act as penalty terms dragging the numerical solution at the boundary towards zero so that the boundary condition is simultaneously approximated. In general, the boundary condition is not satisfied exactly. The penalty parameter $\tau$ is to be determined so that an energy estimate of the discrete solution exists, which ensures stability. 

\subsection{Stability}
In \cite{Appelo2007,Mattsson2008,Mattsson2009}, it is shown that the operator $A$ can be written as
\begin{equation}\label{borrow}
A=h\alpha_{2p}(E_0S)^TE_0S+\tilde{A},
\end{equation}
where $\alpha_{2p}>0$ is as large as possible and $\tilde{A}$ is symmetric positive semi-definite. This is often called the borrowing trick. We use the technique in \cite{Mattsson2008} to compute the values of $\alpha_{2p}$ and list the results in Table \ref{alphavalue}. 
\begin{table}
\centering
\begin{tabular}{c c  c  c c c}
\toprule
$2p$ & 2 & 4 & 6 & 8 & 10\\ \midrule
$\alpha_{2p}$ &  0.4 & 0.2508560249 & 0.1878715026 & 0.0015782259 & 0.0351202265\\ \bottomrule
\end{tabular}
\caption{$\alpha_{2p}$ values}
\label{alphavalue}
\end{table}

Multiplying Equation (\ref{semidiscrete1dhomo}) by $u_t^TH$ from the left, with (\ref{borrow}) we obtain 
\begin{equation*}
\frac{d}{dt}\left(\underbrace{ \|u_t\|_H^2+\|u\|_{\tilde{A}}^2+\left(\sqrt{h\alpha_{2p}} (BSu)_0-\frac{1}{\sqrt{h\alpha_{2p}}}u_0\right)^2+(\tau-\frac{1}{\alpha_{2p}})\frac{u_0^2}{h} }_{E_h}\right)=2u_tHF_g.
\end{equation*}
For $\tau\geq\frac{1}{\alpha_{2p}}$, we have  $E_h\geq 0$.
In this case, $\vertiii{u}^2_h:=E_h$ is a discrete energy, and $\vertiii{\cdot}_h$ is the corresponding discrete energy norm. We then obtain the discrete energy estimate 
\begin{equation*}
\sqrt{E_h}\leq\sqrt{E_{h,0}}+\int_0^t\|F_g(\cdot,z)\|_Hdz,
\end{equation*}
where $E_{h,0}$ is the initial discrete energy. Clearly, stability is ensured if $\tau\geq\tau_{2p}:=\frac{1}{\alpha_{2p}}$.

\begin{remark}
It is possible to include the term $\|u\|_H^2$ in the discrete energy $E_h$. A similar energy estimate holds \cite{Gustafsson2013}. This is useful when comparing the numerical results with the theoretical analysis.
\end{remark}

\subsection{Accuracy analysis by the energy estimate}
Assuming that the continuous problem has a smooth solution $U(x,t)$, we can derive the error equation for the pointwise error $\xi_j=U(x_j,t)-u_j(t)$. The error equation corresponding to (\ref{semidiscrete1dhomo}) is 
\begin{equation}\label{erreqn1d}
\xi_{tt}=D\xi-H^{-1}S^TE_0\xi-\frac{\tau}{h}H^{-1}E_0\xi+T^{2p},
\end{equation}
where $\|\xi(t)\|_h<\infty$, $T^{2p}$ is the truncation error and $2p$ is the accuracy order of the SBP operator. The first $m$ components of $T^{2p}$ are of order $\mathcal{O}(h^p)$, and all the other components are of order $\mathcal{O}(h^{2p})$. For $2p=2,4,6,8,10$, the corresponding $m$ values are $1,4,6,8,11$. In the analysis, we only consider the leading term of the truncation error.
We introduce the interior truncation error $T^{2p,I}$, and the boundary truncation error $T^{2p,B}$ such that $T^{2p}=h^{2p}T^{2p,I}+h^pT^{2p,B}$, where $T^{2p,I}$ and $T^{2p,B}$ are independent of $h$, but depend on the derivatives of $U$. We have
\begin{equation*}
T^{2p,I}_i\begin{cases}
=0 & 0\leq i\leq m-1 \\
\neq 0 & i>m-1
\end{cases}\quad\text{and}\quad
T^{2p,B}_i\begin{cases}
=0 & i> m-1 \\
\neq 0 & 0\leq i\leq m-1
\end{cases}.
\end{equation*}
Since the number of grid points with the large truncation error $\mathcal{O}(h^p)$ is finite and independent of $h$, we have
\begin{equation*}
\|T^{2p}\|_h^2=h(h^{4p}\underbrace{\sum_{i=m}^{\infty}|T^{2p,I}_i|^2}_{~\mathcal{O}(h^{-1})}+h^{2p}\underbrace{\sum_{i=0}^{m-1}|T^{2p,B}_i|^2}_{\mathcal{O}(1)})\leq K_I h^{4p}+K_Bh^{2p+1}.
\end{equation*}
For $2p=2,4,6,8,10$ and small $h$, the first term is much smaller than the second one. Thus, we have $\|T^{2p}\|_h\leq\tilde{K}_B h^{p+1/2}$.
By applying energy method to the error equation (\ref{erreqn1d}), we obtain
\begin{equation*}
\vertiii{\xi}_h\leq C_eh^{p+1/2}.
\end{equation*}
This means that by the energy method we can only prove a gain in accuracy order of 1/2. Therefore, the convergence rate is at least $p+\frac{1}{2}$ if the numerical scheme is stable, that is if $\tau\geq\tau_{2p}$.

\subsection{Normal mode analysis for the boundary truncation error}
To derive a sharp estimate, we partition the pointwise error into two parts, the interior error $\epsilon^I$ and the boundary error $\epsilon$, such that $\xi=\epsilon^I+\epsilon$. $\epsilon^I$ is the error due to the interior truncation error, and $\epsilon$ is the error due to the boundary truncation error. $\epsilon^I$ can be estimated by the energy method, yielding 
\begin{equation}\label{interior_estimate_1d}
\vertiii{\epsilon^I}_h\leq C_I h^{2p}.
\end{equation} 
where $C_I$ is a constant independent of $h$.

We perform a normal mode analysis to analyze the boundary truncation error $\epsilon$ for second, fourth and sixth order SBP-SAT scheme. The boundary error equation is 
\begin{equation}\label{Berreqn1d}
\epsilon_{tt}=D\epsilon-H^{-1}S^TE_0\epsilon-\frac{\tau}{h}H^{-1}E_0\epsilon+h^pT^{2p,B},
\end{equation}
where $\|\epsilon\|_h<\infty$. In the analysis of $\epsilon$, we take the Laplace transform in time of (\ref{Berreqn1d}),
\begin{equation}\label{Berreqn1dLaplace}
s^2\hat{\epsilon}=D\hat{\epsilon}-H^{-1}S^TE_0\hat{\epsilon}-\frac{\tau}{h}H^{-1}E_0\hat{\epsilon}+h^p\hat{T}^{2p,B},
\end{equation}
for $\text{Re}(s)>0$, where $\hat{}$ denotes the variable in the Laplace space. Let $\tilde{s}=sh$. After multipling $h^2$ on both sides of (\ref{Berreqn1dLaplace}), we obtain
\begin{equation}\label{Berreqn1dLaplaceTilde}
\tilde{s}^2\hat{\epsilon}=h^2D\hat{\epsilon}-h^2H^{-1}S^TE_0\hat{\epsilon}-\tau hH^{-1}E_0\hat{\epsilon}+h^{p+2}\hat{T}^{2p,B}.
\end{equation}
Note that the first three terms in the right hand side of (\ref{Berreqn1dLaplaceTilde}) are $h$-independent. 

By assumption the true solution $U(x,t)$ is smooth and, consequently, $|\hat{T}^{2p,B}(s)|$ decreases fast for large $|s|$. Convergence rate is an asymptotic behaviour as the mesh size $h$ approaches zero. Therefore, we need to investigate the solution of (\ref{Berreqn1dLaplaceTilde}) as $\tilde s$ approaches zero. In the analysis, we assume $|s|\leq K$ and consider $\text{Re}(s)\geq\eta>0$, where $K$ and $\eta$ are some positive constants. Equivalently, $|\tilde s|\ll 1$ and $\text{Re}(\tilde s)\geq\eta h$. If the semi-discretized equation is stable, the Laplace transformed problem is nonsingular for Re$(\tilde s)>0$ \cite{Gustafsson2013}. 

There are essentially two steps in the normal mode analysis. Firstly, by a detailed analysis of the error equation (\ref{Berreqn1dLaplaceTilde}), we obtain an estimate for $\|\hat\epsilon\|_h$ in the Laplace space.
Then we use Parseval's relation to derive an estimate for the error in the physical space of the form
\begin{equation}\label{qestimate}
\sqrt{\int_0^{t_f}\|\epsilon(\cdot,t)\|^2_hdt}\leq Ch^q,
\end{equation}
where $C$ depends only on $\eta$, the final time $t_f$ and the derivatives of the true solution $U$. 
The results for the Dirichlet problem are summarized in Theorem \ref{1dresult}.
\begin{theorem}\label{1dresult}
For the second, fourth and sixth order stable SBP-SAT approximations of the second order wave equation (\ref{1dwave1}-\ref{1dwave2}) on a half line with Dirichlet boundary condition, the rates $q$ in (\ref{qestimate}) depend on $\tau$, and are listed in Table \ref{1d_table}.
\begin{table}
\centering
\begin{tabular}{c c c}
\toprule
$2p$ & $q({\tau>\tau_{2p}})$ & $q(\tau=\tau_{2p})$ \\ 
\midrule
2 & 2 & 1.5 \\
4 & 4 & 2.5 \\
6 & 5 & 3.5 \\
\bottomrule
\end{tabular}
\caption{Theoretical convergence result for one dimensional wave equation.}
\label{1d_table}
\end{table}
\end{theorem}

In \S\ref{numericalexperiments}, we present results from numerical experiments, which agree well with the theoretical results in Table \ref{1d_table}.
After some preliminaries we will prove Theorem \ref{1dresult} for $2p=2,4$ and 6 in \S\ref{1d2o}, \ref{1d4o} and \ref{1d6o}, respectively. In the proof, we will show that the determinant condition is satisfied if $\tau>\tau_{2p}$ but not satisfied if $\tau=\tau_{2p}$. In addition, the determinant condition is necessary for the optimal convergence rate. 

To begin with, we note that sufficiently far away from the boundary, the two penalty terms and the boundary truncation error in (\ref{Berreqn1dLaplaceTilde}) are not present. The coefficients in $D$ correspond to that in the standard central finite difference scheme. We have
\begin{subequations}\label{interior_equation}
\begin{align}
&2p=2:\quad s^2\hat{\epsilon}_j=D_+D_-\hat{\epsilon}_j,\ j=3,4,5,\cdots \label{interior_equation1}, \\
&2p=4:\quad s^2\hat{\epsilon}_j=(D_+D_--\frac{1}{12}(D_+D_-)^2)\hat{\epsilon}_j,\ j=4,5,6,\cdots \label{interior_equation2}, \\
&2p=6:\quad s^2\hat{\epsilon}_j=(D_+D_--\frac{h^2}{12}(D_+D_-)^2+\frac{h^4}{90}(D_+D_-)^3)\hat{\epsilon}_j,\ j=5,6,7,\cdots \label{interior_equation3},
\end{align}
\end{subequations}
where 
\begin{equation*}
D_+\hat{\epsilon}_j=\frac{\hat{\epsilon}_{j+1}-\hat{\epsilon}_j}{h}\quad \text{and}\quad D_-\hat{\epsilon}_j=\frac{\hat{\epsilon}_j-\hat{\epsilon}_{j-1}}{h}.
\end{equation*}
The corresponding characteristic equations are
\begin{subequations}\label{characteristic_equation}
\begin{align}
&2p=2:\quad \kappa^2-(2+\tilde{s}^2)\kappa+1=0\label{characteristic_equation1}, \\
&2p=4:\quad \kappa^4-16\kappa^3+(30+12\tilde{s}^2)\kappa^2-16\kappa+1=0\label{characteristic_equation2}, \\
&2p=6:\quad 2\kappa^6-27\kappa^5+270\kappa^4-(180\tilde{s}^2+490)\kappa^3+270\kappa^2-27\kappa+2=0.\label{characteristic_equation3}
\end{align}
\end{subequations}
It is easy to verify by von Neumann analysis that the interior numerical scheme is stable when applied to the corresponding periodic problem. From Lemma 12.1.3 in \cite[pp.~379]{Gustafsson2013}, it is straightforward to prove that there is no root with $|\kappa|=1$ for $\text{Re}(s)>0$. We will need the following specifics for the roots.
\begin{lemma}\label{lemma_characteristic_solution}
For $2p=2,4,6$, the number of admissible roots of (\ref{characteristic_equation}) satisfying $|\kappa|<1$ for $\text{Re}(\tilde{s})>0$ is 1,2,3, respectively. In the vicinity of $\tilde{s}=0$, they are given by
\begin{subequations}\label{characteristic_solution}
\begin{align}
2p=2:\quad &\kappa_1=1-\tilde{s}+\mathcal{O}(\tilde{s}^2), \label{characteristic_solution1}\\
2p=4:\quad &\kappa_1=1-\tilde{s}+\mathcal{O}(\tilde{s}^2), \ \kappa_2=7-4\sqrt{3}+\mathcal{O}(\tilde{s}^2), \label{characteristic_solution2}\\
2p=6:\quad &\kappa_1=1-\tilde{s}+\mathcal{O}(\tilde{s}^2),\ \kappa_2=0.0519-0.0801i+\mathcal{O}(\tilde{s}^2),\label{characteristic_solution3} \\
 &\kappa_3=0.0519+0.0801i+\mathcal{O}(\tilde{s}^2).  \nonumber
\end{align}
\end{subequations}
\end{lemma}

\begin{proof}
$2p=2$: Equation (\ref{characteristic_equation1}) has two roots:
$\kappa_{1,2}=1+\frac{1}{2}\tilde{s}^2\pm\frac{1}{2}\sqrt{\tilde{s}^4+4\tilde{s}^2}$. We find by Taylor expansion at $\tilde{s}=0$ that $\kappa_1=1-\tilde{s}+\mathcal{O}(\tilde{s}^2)$ and $\kappa_2=1+\tilde{s}+\mathcal{O}(\tilde{s}^2)$. Thus, the admissible root is $\kappa_1=1-\tilde{s}+\mathcal{O}(\tilde{s}^2)$. 

$2p=4$: Equation (\ref{characteristic_equation2}) has four roots:
\begin{equation*}
\begin{split}
\kappa_1&=\sqrt{24-3\tilde{s}-8\sqrt{9-3\tilde{s}^2}}-\sqrt{9-3\tilde{s}^2}+4,\\
\kappa_2&=4-\sqrt{9-3\tilde{s}^2}-\sqrt{24-3\tilde{s}-8\sqrt{9-3\tilde{s}^2}},  \\
\kappa_3&=\sqrt{9-3\tilde{s}^2}-\sqrt{8\sqrt{9-3\tilde{s}^2}-3\tilde{s}+24}+4,\\
\kappa_4&=\sqrt{8\sqrt{9-3\tilde{s}^2}-3\tilde{s}+24}+\sqrt{9-3\tilde{s}^2}+4. 
\end{split}
\end{equation*}
We find by Taylor expansion at $\tilde{s}=0$ that
\begin{equation*}
\begin{split}
\kappa_1&=1-\tilde{s}+\mathcal{O}(\tilde{s}^2),\quad\quad\quad \kappa_2=7-4\sqrt{3}+\mathcal{O}(\tilde{s}^2), \\
\kappa_3&=7+4\sqrt{3}+\mathcal{O}(\tilde{s}^2),\quad\ \ \kappa_4=1+\tilde{s}+\mathcal{O}(\tilde{s}^2). 
\end{split}
\end{equation*}
Thus, the admissible roots are $\kappa_1=1-\tilde{s}+\mathcal{O}(\tilde{s}^2)$ and $\kappa_2=7-4\sqrt{3}+\mathcal{O}(\tilde{s}^2)$. 

$2p=6$: (\ref{characteristic_equation3}) is a six order equation. There is no formula for general six order equations over the rationals in terms of radicals. We solve numerically equation (\ref{characteristic_equation3}) with $\tilde{s}=0$, and then analyze the six roots by perturbation theory to find out their dependence on $\tilde{s}$. The three admissible roots are given by (\ref{characteristic_solution3}).
\end{proof}

\begin{lemma}\label{kappa1}
Consider $|\tilde{s}|\ll 1 $ and $\text{Re}(\tilde{s})\geq\eta h>0$. For $2p=2,4,6$ the admissible roots $\kappa_1(\tilde{s})$ satisfy
\begin{equation*}
1-|\kappa_1(\tilde{s})|^2\geq 2\text{Re}(\tilde{s})+\mathcal{O}(|\tilde{s}|^2).
\end{equation*}
\end{lemma}
\begin{proof}
Let $\tilde{s}=x+iy$ where $x,y$ are real numbers. Then $|x|\geq\eta h$.
\begin{equation*}
\begin{split}
1-|\kappa_1(\tilde{s})|^2 &= 1-|1-\tilde{s}+\mathcal{O}(\tilde{s}^2)|^2 \\
&\geq 1-|1-\tilde{s}|^2+\mathcal{O}(|\tilde{s}|^2) \\
&= 1-|1-x-yi|^2+\mathcal{O}(|\tilde{s}|^2) \\
&= 2x-x^2-y^2+\mathcal{O}(|\tilde{s}|^2) \\
&= 2\text{Re}(\tilde{s})+\mathcal{O}(|\tilde{s}|^2).
\end{split}
\end{equation*}
\end{proof}
By Lemma \ref{kappa1}, we obtain $\frac{1}{1-|\kappa (\tilde s)|^2} \leq \frac{1}{2\eta h}$ to the leading order. We will use this inequality to estimate $\|\hat\epsilon\|_h$.

\subsubsection{Proof of Theorem \ref{1dresult} for the second order scheme}\label{1d2o}
Away from the boundary, equation (\ref{Berreqn1dLaplaceTilde}) is simplified to (\ref{interior_equation1}), which is a recursion relation for $\hat{\epsilon}_j$. A general solution of (\ref{interior_equation1}) can be written as
\begin{equation}\label{1dsbp2i}
\hat{\epsilon}_j=\sigma\kappa_1^{j-2},\ j=2,3,4,\cdots,
\end{equation}
where $\kappa_1$ is obtained from (\ref{characteristic_solution1}) and $\sigma$ is an unknown in the boundary system. In the following, we derive the boundary system and investigate its solution near $\tilde s=0$ to obtain an estimate of the error in the Laplace space. 
The first three rows of (\ref{Berreqn1dLaplaceTilde}) are affected by the penalty terms. They are
\begin{equation}\label{2ndr3}
\begin{split}
\tilde{s}^2\hat{\epsilon}_0&=\hat{\epsilon}_0-2\hat{\epsilon}_1+\hat{\epsilon}_2+3\hat{\epsilon}_0-2\tau\hat{\epsilon}_0+h^3\hat{T}^{2,B}_0, \\
\tilde{s}^2\hat{\epsilon}_1&=\hat{\epsilon}_0-2\hat{\epsilon}_1+\hat{\epsilon}_2-2\hat{\epsilon}_0, \\
\tilde{s}^2\hat{\epsilon}_2&=\hat{\epsilon}_1-2\hat{\epsilon}_2+\hat{\epsilon}_3-\frac{1}{2}\hat{\epsilon}_0. 
\end{split}
\end{equation}
By Taylor expansion, it is straightforward to derive $\hat{T}^{2,B}_0=-\hat{U}_{xxx}(0,s)$ to the leading order.
We write (\ref{2ndr3}) in the matrix vector multiplication form with the help of (\ref{1dsbp2i}), and obtain the boundary system
\begin{equation*}
\underbrace{\begin{bmatrix}
-1 & \tilde{s}^2-4+2\tau & 2 \\
-1 & 1 & \tilde{s}^2+2 \\
\kappa_1-2-\tilde{s}^2 & \frac{1}{2} & 1
\end{bmatrix}}_{C_{2D}(\tilde{s},\tau)}
\underbrace{\begin{bmatrix}
\sigma \\
\hat{\epsilon}_0 \\
\hat{\epsilon}_1
\end{bmatrix}}_{\Sigma_{2D}}=h^3\underbrace{\begin{bmatrix}
-\hat{U}_{xxx}(0,s) \\
0 \\
0
\end{bmatrix}}_{\hat{T}^{2,B}_v}.
\end{equation*}
Next, we investigate the invertibility of $C_{2D}(\tilde{s},\tau)$ in a vicinity of $\tilde{s}=0$ to derive an estimate for $|\Sigma_{2D}|$, where $|\cdot|$ denotse the standard Euclidean norm of vectors and matrices. We note that
\begin{equation*}
C_{2D}(0,\tau)=\begin{bmatrix}
-1 & -4+2\tau & 2\\
-1 & 1 & 2 \\
-1 & \frac{1}{2} & 1
\end{bmatrix}.
\end{equation*}
The determinant of $C_{2D}(0,\tau)$ is given by $\text{det}(C_{2D}(0,\tau))=5-2\tau$. Clearly, $C_{2D}(0,\tau)$ is singular if $\tau=2.5$ and nonsingular otherwise. The stability condition is $\tau\geq\tau_2=2.5$.

In the case $\tau>\tau_2$, $C_{2D}(0,\tau)$ is nonsingular. By the continuity of $C_{2D}(\tilde{s},\tau)$, for any $\gamma>1$ and sufficiently small $h$ we have $|C_{2D}^{-1}(\tilde{s},\tau)| \leq \gamma |C_{2D}^{-1}(0,\tau)|$.
Thus, $|\Sigma_{2D}|\leq \gamma |C_{2D}^{-1}(0,\tau)| |\hat{T}^{2,B}_v| h^3$. Therefore, we have $|\sigma|^2 \leq K_\sigma |\hat{U}_{xxx}(0,s)|^2 h^6$ and $\sum_{i=0}^1|\hat\epsilon_i|^2\leq K_\epsilon  |\hat{U}_{xxx}(0,s)|^2 h^6$, where $K_\sigma$ and $K_\epsilon$ are constants independent of $h$.
In $\text{L}_2$ norm, we have
\begin{equation*}
\|\hat\epsilon\|^2_h=h\sum_{i=0}^1|\hat\epsilon_i|^2+h|\sigma|^2\sum_{i=0}^\infty|\kappa_1|^{2i}=\underbrace{h\sum_{i=0}^1|\hat\epsilon_i|^2}_{\|\hat\epsilon\|^2_{B,h}}+\underbrace{h|\sigma|^2\frac{1}{1-|\kappa_1|^2}}_{\|\hat\epsilon\|^2_{I,h}}.
\end{equation*}
By Lemma \ref{kappa1}, we have $\|\hat\epsilon\|^2_{I,h}=h|\sigma|^2\frac{1}{1-|\kappa_1|^2}\leq \frac{K_\sigma}{2\eta}|\hat{U}_{xxx}(0,s)|^2 h^6$. The first term can be bounded as $\|\hat\epsilon\|^2_{B,h}=h\sum_{i=0}^1|\hat\epsilon_i|^2\leq K_\epsilon |\hat{U}_{xxx}(0,s)|^2 h^7\ll\|\hat\epsilon\|^2_{I,h}$ for small $h$. Thus, to the leading order,
\begin{equation*}
\|\hat\epsilon\|^2_h\leq \frac{\tilde K_\sigma}{2\eta } |\hat{U}_{xxx}(0,s)|^2 h^6,
\end{equation*}
where  $\tilde K_\sigma$ is a constant independent of $h$. By Parseval's relation, we have 
\begin{equation*}
\begin{split}
\int_0^\infty e^{-2\eta t}\|\epsilon\|^2_hdt&=\frac{1}{2\pi}\int_{-\infty}^{\infty}\|\hat{\epsilon}(\eta+i\xi)\|^2_h d\xi \\
&\leq \frac{\tilde K_\sigma h^6}{4\pi\eta}\int_{-\infty}^{\infty}|\hat{U}_{xxx}(0,\eta+i\xi)|^2d\xi \\
&=\frac{\tilde K_\sigma h^6}{2\eta}\int_0^\infty e^{-2\eta t}|U_{xxx}(0,t)|^2dt.
\end{split}
\end{equation*}
Note that in the above, we use Parseval's relation twice. By arguing that the future cannot affect the past, we can replace the upper limit of the integrals on both sides by a finite time $t_f$. Since 
\begin{equation*}
\begin{split}
&\int_0^{t_f} e^{-2\eta t_f}\|\epsilon\|^2_hdt\leq \int_0^{t_f} e^{-2\eta t}\|\epsilon\|^2_hdt, \\
&\int_0^{t_f} e^{-2\eta t}|U_{xxx}(0,t)|^2dt\leq \int_0^{t_f} |U_{xxx}(0,t)|^2dt,
\end{split}
\end{equation*}
 we obtain
\begin{equation*}
\sqrt{\int_0^{t_f} \|\epsilon\|^2_hdt} \leq h^3\sqrt{\frac{\tilde K_\sigma e^{2\eta t_f} }{2\eta}\int_0^{t_f} |U_{xxx}(0,t)|^2dt}.
\end{equation*}
Thus, the boundary error is $\mathcal{O}(h^3)$. 
In this case, the interior error is $\mathcal{O}(h^2)$ by (\ref{interior_estimate_1d}), which is the dominating source of error. Therefore, the overall convergence rate is 2.

In the case $\tau=\tau_2$, $C_{2D}(0,\tau)$ is singular, and the above analysis is not valid. By the energy estimate the convergence rate is $p+\frac{1}{2}=1.5$. 

\subsubsection{Proof of Theorem \ref{1dresult} for the fourth order scheme}\label{1d4o}
Away from the boundary, the solution of (\ref{Berreqn1dLaplaceTilde}) is
\begin{equation}\label{1dsbp4i}
\hat{\epsilon}_j=\sigma_1\kappa_1^{j-2}+\sigma_2\kappa_2^{j-2},\ j=2,3,4,\cdots,
\end{equation}
where $\kappa_1,\kappa_2$ are obtained from (\ref{characteristic_solution2}), and $\sigma_1,\sigma_2$ will be determined by the boundary system.
The first four rows of (\ref{Berreqn1dLaplaceTilde}) are affected by the penalty terms. They are
\begin{equation}\label{4thr4}
\begin{split}
\tilde{s}^2\hat{\epsilon_0}&=\left(\frac{122-48\tau}{17}\right)\hat{\epsilon}_0-5\hat{\epsilon}_1+4\hat{\epsilon}_2-\hat{\epsilon}_3+h^4\hat{T}_0^{4,B}, \\
\tilde{s}^2\hat{\epsilon}_1&=-\frac{85}{59}\hat{\epsilon}_0-2\hat{\epsilon}_1+\hat{\epsilon}_2+h^4\hat{T}_1^{4,B}, \\
\tilde{s}^2\hat{\epsilon}_2&=\frac{68}{43}\hat{\epsilon}_0+\frac{59}{43}\hat{\epsilon}_1-\frac{110}{43}\hat{\epsilon}_2+\frac{59}{43}\hat{\epsilon}_3-\frac{4}{43}\hat{\epsilon}_4+h^4\hat{T}_2^{4,B}, \\
\tilde{s}^2\hat{\epsilon}_3&=-\frac{17}{49}\hat{\epsilon}_0+\frac{59}{49}\hat{\epsilon}_2-\frac{118}{49}\hat{\epsilon}_3+\frac{64}{49}\hat{\epsilon}_4-\frac{4}{49}\hat{\epsilon}_5+h^4\hat{T}_3^{4,B}.
\end{split}
\end{equation}
By Taylor expansion, we have 
\begin{equation*}
\begin{split}
&\hat{T}_0^{4,B}=\frac{11}{12}\hat{U}_{xxxx}(0,s),\quad \hat{T}_1^{4,B}=-\frac{1}{12}\hat{U}_{xxxx}(0,s), \\
&\hat{T}_2^{4,B}=\frac{5}{516}\hat{U}_{xxxx}(0,s),\quad \hat{T}_3^{4,B}=\frac{11}{588}\hat{U}_{xxxx}(0,s)
\end{split}
\end{equation*}
to the leading order. 
With $\tilde{s}=0$, the boundary system follows from (\ref{1dsbp4i})--(\ref{4thr4})
\begin{equation*}
C_{4D}(0,\tau)\Sigma_{4D}=h^4\hat{T}_v^{4,B},
\end{equation*}
where $\Sigma_{4D}=[\sigma_1, \sigma_2, \hat{\epsilon}_0, \hat{\epsilon}_1]^T$,
$\hat{T}_v^{4,B}=[\hat{T}_0^{4,B}, \hat{T}_1^{4,B}, \hat{T}_2^{4,B}, \hat{T}_3^{4,B}]^T$
and $C_{4D}(0,\tau)$ is given in Appendix 1.
We solve for $\text{det}(C_{4D}(0,\tau))=0$, and obtain $\tau=\frac{2(4834\sqrt{3}+9569)}{177(8\sqrt{3}+37)}\approx 3.986350342$. Recall that the stability condition is $\tau\geq\tau_4\approx 3.986350342$.

In the case $\tau>\tau_4$, $C_{4D}(0,\tau)$ is nonsingular. Similarly to the second order case in \S\ref{1d2o}, we obtain to leading order
\begin{equation*}
\|\hat\epsilon\|^2_h\leq \frac{\tilde K_{\sigma_1}}{2\eta} |\hat{U}_{xxxx}(0,s)|^2 h^8.
\end{equation*}
By Parseval's relation and by arguing that the future cannot affect the past, we obtain
\begin{equation*}
\sqrt{\int_0^{t_f} \|\epsilon\|^2_hdt} \leq h^4\sqrt{\frac{\tilde K_{\sigma_1} e^{2\eta t_f} }{2\eta}\int_0^{t_f} |U_{xxxx}(0,t)|^2dt}.
\end{equation*}
Thus, the boundary error is $\mathcal{O}(h^4)$. In this case, the interior error is also $\mathcal{O}(h^4)$ given by (\ref{interior_estimate_1d}). Therefore, the convergence rate is 4.

In the case $\tau=\tau_4$, $C_{4D}(0,\tau_4)$ is singular. By the energy estimate, the convergence rate is $p+\frac{1}{2}=2.5$.

\subsubsection{Proof of Theorem \ref{1dresult} for the sixth order scheme}\label{1d6o}
Away from the boundary, the solution of (\ref{Berreqn1dLaplaceTilde}) is
\begin{equation}\label{1dsbp6i}
\hat{\epsilon}_j=\sigma_1\kappa_1^{j-3}+\sigma_2\kappa_2^{j-3}+\sigma_3\kappa_3^{j-3},\ j=3,4,5,\cdots,
\end{equation}
where $\kappa_1,\kappa_2,\kappa_3$ are obtained from (\ref{characteristic_solution3}), and $\sigma_1,\sigma_2,\sigma_3$ are determined by the boundary system. 
Near the boundary, the first six rows of (\ref{Berreqn1dLaplaceTilde}) are affected by the penalty terms. In the same way as for the second and fourth order method, we obtain the boundary system in $\Sigma_{6D}=[\hat{\epsilon}_0,\hat{\epsilon}_1,\hat{\epsilon}_2,\sigma_1,\sigma_2,\sigma_3]^T$:
\begin{equation*}
C_{6D}(\tilde{s},\tau)\Sigma_{6D}=h^5\hat{T}_v^{6,B},
\end{equation*}
where $C_{6D}(\tilde{s},\tau)$ is given in Appendix 1. We find that $C_{6D}(0,\tau)$ is singular if $\tau=\tau_6$, and nonsingular otherwise. When $\tau>\tau_6$, we use Taylor expansion to express $\hat{T}_v^{6,B}$ in terms of the derivatives of the true solution $U$, and Lemma \ref{kappa1} to derive an estimate for (\ref{1dsbp6i}). The boundary error is $\mathcal{O}(h^5)$. In this case, the interior error is $\mathcal{O}(h^6)$ given by (\ref{interior_estimate_1d}). Thus, the convergence rate is 5.

When $\tau=\tau_6$, the convergence rate $p+\frac{1}{2}=3.5$ is given by the energy estimate. 

\section{One dimensional wave equation with Neumann boundary condition}\label{section_Neumann}
Next, we consider an example which never satisfies the determinant condition, but nonetheless exhibits optimal convergence. We consider Equation (\ref{1dwave1}) with Neumann boundary condition
\begin{equation}\label{NeumannBC}
U_x(0,t)=g(t).
\end{equation}
We use the same assumption of the data as for the Dirichlet problem. With homogeneous Neumann boundary condition $g(t)=0$ we get an energy estimate identical to the one for the Dirichlet problem (\ref{1denergy}). To prove well-posedness, we also need to show that the equation is boundary stable. Boundary stability for the Neumann problem is proved in \cite[Theorem 3.8]{Kreiss2012} by Laplace transform technique. 

\subsection{Stability}
To discretize the equation in space, we use the grid and grid functions introduced for the Dirichlet problem. The semi--discretized equation of the Neumann problem is 
\begin{equation}\label{semi-Neumann}
u_{tt}=Du+H^{-1}E_0Su+F_g.
\end{equation}
On the right hand side of (\ref{semi-Neumann}), the first term approximates $U_{xx}$ and the second term imposes weakly the boundary condition $U_x(0,t)=0$. Since we consider a half line problem, the SBP operator $D$ can be written as $D=H^{-1}(-M-E_0S)$. As a consequence, the semi--discretization (\ref{semi-Neumann}) can be written as
\begin{equation}\label{semi-Neumann2}
u_{tt}=-H^{-1}Mu+F_g.
\end{equation}

Multiplying Equation (\ref{semi-Neumann2}) by $u_t^TH$ from the left, we obtain
\begin{equation*}
\frac{d}{dt}\left( \|u_t\|_H^2+\|u\|_M^2 \right)=2u_tHF.
\end{equation*}
The discrete energy $E_h^{Neu}=\vertiii{u}_h^{Neu}:= \|u_t\|_H^2+\|u\|_M^2 $ is bounded as
\begin{equation}\label{energy_Neumann}
\sqrt{E_h^{Neu}}\leq\sqrt{E_{h,0}^{Neu}}+\int_0^t \|F_g(\cdot,z)\|_Hdz.
\end{equation} 
(\ref{energy_Neumann}) is the energy estimate. 

\subsection{Accuracy}
Define the pointwise error as $\xi_j=U(x_j,t)-u_j(t)$. The error equation corresponding to (\ref{semi-Neumann2}) is 
\begin{equation}\label{err_eqn_Neumann}
\xi_{tt}=-H^{-1}M\xi+T^{2p,Neu},
\end{equation}
where $\|\xi\|_h<\infty$, $T^{2p,Neu}$ is the truncation error and $2p$ is the accuracy order of the SBP operator. In the same way as for the Dirichlet problem, by applying the energy method to the error equation (\ref{err_eqn_Neumann}), we obtain an estimate
\begin{equation*}
\vertiii{\xi}_h^{Neu}\leq C_{Neu} h^{p+1/2},
\end{equation*}
where $C_{Neu}$ is a constant independent of $h$. This means that the convergence rate of (\ref{semi-Neumann2}) is at least $p+1/2$. In the following, we use normal mode analysis to derive a sharp bound of the error, which agrees well with the results from numerical experiments.

Similar to the Dirichlet problem, we partition the error $\xi$ into two parts, the error $\epsilon^I$ due to the interior truncation error and the error $\epsilon$ due to the boundary truncation error. $\epsilon^I$ can be estimated by the energy method, yielding $\vertiii{\epsilon^I}_h^{Neu}\leq C_I^{Neu} h^{2p}$ where $C_I^{Neu}$ is a constant independent of $h$. The boundary error equation is
\begin{equation}\label{b_err_eqn_Neumann}
\epsilon_{tt}=-H^{-1}M\epsilon+h^pT^{2p,Neu,B}
\end{equation}
where $h^pT^{2p,Neu,B}$ is the boundary truncation error. $T^{2p,Neu,B}$ depends on the derivatives of the exact solution $U$, but not $h$. Moreover, only the first $m$ elements of $T^{2p,Neu,B}$ are nonzero, where $m$ depends on $p$ but not $h$. 

We Laplace transform Equation (\ref{b_err_eqn_Neumann}), with the notation $\tilde s=sh$, we obtain
\begin{equation}\label{b_err_eqn_Neumann_L}
{\tilde s}^2\hat\epsilon=-h^2H^{-1}M\hat\epsilon+h^{p+2}\hat T^{2p,Neu,B}
\end{equation}
Comparing with the error equation for the Dirichlet problem (\ref{Berreqn1dLaplaceTilde}), there are less rows affected by the boundary closure than that for the Drichlet problem. The characteristic equations of the interior error equations are the same as for the Dirichlet problem (\ref{characteristic_equation}). Lemma \ref{lemma_characteristic_solution} and \ref{kappa1} are also applicable to the Neumann problem. For the error equation near the boundary, we analyze below the second, fourth and sixth order accurate cases. 
\subsubsection{Second order accurate scheme}
For second order accurate scheme, only the first row of the error equation is affected by the boundary closure. The general solution of the error equation is
\begin{equation}\label{interior_err_2nd}
\hat\epsilon_j=\sigma\kappa_1^j,\ j=0,1,2,\cdots,
\end{equation}
where $\kappa_1=1-\tilde s+\mathcal{O}(\tilde s^2)$ is obtained from (\ref{characteristic_solution1}) and $\sigma$ will be determined by the boundary system. The first row of (\ref{b_err_eqn_Neumann_L}) is 
\begin{equation*}
{\tilde s}^2 \hat\epsilon_0 = -2 \hat\epsilon_0 + 2 \hat\epsilon_1 + h^3 \hat T_0^{2,Neu,B}.
\end{equation*}
From Equation (\ref{interior_err_2nd}), we obtain
\begin{equation}
C_{2N}(\tilde s)\sigma=h^3 \hat T_0^{2,Neu,B},\ C_{2N}(\tilde s)={\tilde s}^2 +2-2\kappa_1=2\tilde s+\mathcal{O}(\tilde s^2).
\end{equation}
Clearly, $C_{2N}(0)=0$. By Taylor expansion, we have
\begin{equation*}
\hat T_0^{2,Neu,B}=-\hat U_{xxx}(0,s)+\frac{1}{3}\hat U_{xxx}(0,s)=-\frac{2}{3}\hat U_{xxx}(0,s).
\end{equation*}
where $-h^3\hat U_{xxx}(0,s)$ is the truncation error by the SBP operator and $h^3\hat U_{xxx}(0,s)/3$ is the truncation error by the SAT.
Therefore, we have $|\sigma|\leq\frac{C}{|s|}|\hat U_{xxx}(0,s)|h^2$. In the same way as for the Dirichlet problem, an estimate in physical space follows, and we find that the boundary error is $\mathcal{O}(h^2)$. Since the interior error is also $\mathcal{O}(h^2)$, the convergence rate is 2. For second order accurate scheme, we gain two orders from the boundary truncation error in the Dirichlet problem, but only gain one order in the Neumann problem. Fortunately, this does not affect the overall convergence rate.

\subsubsection{Fourth and sixth order accurate scheme}\label{sec-Neumann-46}
We follow the above approach to derive the four by four boundary system 
\begin{equation*}
C_{4N}(\tilde s)\Sigma_{4N}=h^4 \hat T_u^{4,Neu,B},
\end{equation*}
for the fourth order scheme. At $\tilde s=0$ we find that  $C_{4N}(0)$ is singular with rank 3. By Taylor expansion, we derive the precise form of $\hat T_u^{4,Neu,B}$. 
To derive a sharp estimate for $|\Sigma_{4N}|$, we write $C_{4N}(\tilde s)=C_{4N}(0)+\tilde s C'_{4N}(\tilde s)+\mathcal{O}(\tilde s^2)$ and use Lemma 3.4 in \cite{Nissen2012}. Let singular value decomposition of $C_{4N}(0)$ be $U^*_{4N}S_{4N}V_{4N}$. We find that $U^*_{4N}C'_{4N}(0)V_{4N}=-0.3095\neq 0$. Moreover, $\hat T_u^{4,Neu,B}$ is in the column space of $C_{4N}(0)$. Hence, $|\Sigma_{4N}|\leq K_{4N}|\hat U_{xxxx}(0,s)| h^4$ where $K_{4N}$ is a constant independent of $h$. This means that the boundary error is $\mathcal{O}(h^4)$ and there is a gain of 2 in accuracy order. In this case, the interior error is also $\mathcal{O}(h^4)$, and the overall  convergence rate is 4. We show $C_{4N}(0)$, $C'_{4N}(0)$ and $\hat T_u^{4,Neu,B}$ in Appendix 2.

The situation of the sixth order scheme is very similar to that of the fourth order scheme. The six by six boundary system
\begin{equation*}
C_{6N}(\tilde s)\Sigma_{6N}=h^5 \hat T_u^{6,Neu,B}
\end{equation*}
is singular at $\tilde s=0$ with rank 5. The exact form of $T_u^{6,Neu,B}$ is derived by Taylor expansion.
Here again we find that Lemma 3.4 in \cite{Nissen2012} is applicable with $U^*_{6N}C'_{6N}(0)V_{6N}=0.2072 - 0.0001i\neq 0$, and $\hat T_u^{6,Neu,B}$ is in the column space of $C_{6N}(0)$.
As a consequence, the boundary error is $\mathcal{O}(h^5)$. Since the interior error is $\mathcal{O}(h^6)$, the convergence rate is 5. $C_{6N}(0)$, $C'_{6N}(0)$ and $\hat T_u^{6,Neu,B}$ are given in Appendix 2.
\begin{remark}
For fourth and sixth order scheme, the optimal $p+2$ convergence rate relies on the fact that the boundary truncation error vector is in the column space of $C_{4/6N}(0)$. In the numerical experiments in \S\ref{numericalexperiments}, we will verify the optimal convergence rate. If the boundary truncation error would not be in the column space of $C_{4/6N}(0)$, we only obtain $p+1$ convergence rate by using Lemma 3.4 in \cite{Nissen2012}. In the numerical experiments, we will add a dissipative term to the boundary part of the SBP operator, so that the boundary truncation error is not in the column space of $C_{4/6N}(0)$ but remains the same magnitude. We indeed observe $(p+1)^{th}$ order convergence rate.
\end{remark}

\section{One dimensional wave equation with a grid interface}\label{section_interface}
In this section, we consider another example that does not satisfy the determinant condition. It is the Cauchy problem for the second order wave equation in one dimension
\begin{equation}\label{wave_2d}
\begin{split}
&U_{tt} = U_{xx} + F,\ x\in (-\infty,\infty),\ t\geq 0,\ \|U(\cdot,t)\|<\infty\\
&U(x,0)=f^1(x),\ U_t(x,0)=f^2(x), 
\end{split}
\end{equation}
where the forcing function $F$ , the initial data $f^1$ and $f^2$ are compatible smooth functions with compact support. It is straightforward to derive the energy estimate of the form (\ref{1denergy}).

We solve the equation on a grid with a grid interface at $x=0$. With the assumption that the exact solution is smooth, it is natural to impose two interface conditions at $x=0$: continuity of the solution and continuity of the first derivative of the solution.

We introduce the grid on the left $x_{-j}=-jh_L,\ j=0,1,2,\cdots$, and the grid on the right $x_{j}=jh_R,\ j=0,1,2,\cdots$, with the grid size $h_L$ and $h_R$, respectively. The grid interface is located at $x=0$. The grid functions are $u_j^L(t)\approx U(x_{-j},t)$ and $u_j^R(t)\approx U(x_j,t)$. Denote 
\begin{equation*}
\begin{split}
& e_{0L}=[\cdots,0,0,1]^T,\quad e_{0R}=[1,0,0,\cdots]^T,\\
& u^L=[\cdots,u_{-1}^L,u_0^L]^T,\quad u^R=[u_0^R,u_1^R,\cdots]^T.
\end{split}
\end{equation*}
Both $u_0^L=e_{0L}^Tu^L$ and $u_0^R=e_{0R}^Tu^R$ approximate $U(0,t)$. Both $(S_Lu^L)_0=e_{0L}^TS_Lu^L$ and $(S_Ru^R)_0=e_{0R}^TS_Ru^R$ approximate $U_x(0,t)$. To simplify notation, we define $\{u\}=u_0^L-u_0^R$ and $\{Su\}=(S_Lu^L)_0-(S_Ru^R)_0$. The semi-discretized equation reads
\begin{equation}\label{1dintsemi}
\begin{split}
u_{tt}^L&=D_Lu^L+\frac{1}{2}H_L^{-1}S_L^Te_{0L}\{u\}-\frac{1}{2}H_L^{-1}e_{0L}\{Su\}-\tau_LH_L^{-1}e_{0L}\{u\}+F_{gL}, \\
u_{tt}^R&=D_Ru^R+\frac{1}{2}H_R^{-1}S_R^Te_{0R}\{u\}-\frac{1}{2}H_R^{-1}e_{0R}\{Su\}-\tau_RH_R^{-1}e_{0R}\{u\}+F_{gR},
\end{split}
\end{equation}
where $F_{gL}$ and $F_{gR}$ are the grid functions corresponding to $F(x,t)$. The penalty parameters $\tau_L$ and $\tau_R$ are chosen so that the semi-discretization is stable. By the energy method, the numerical scheme is stable if $\tau_L=-\tau_R\geq\frac{h_L+h_R}{4\alpha_{2p}h_Lh_R}:=\tilde{\tau}_{2p}$. The stability proof is found in \cite[Lemma 2.4, pp.~215]{Mattsson2008}. If $h_L=h_R:=h$, then $\tilde \tau_{2p}=\frac{1}{2\alpha_{2p}h}$.

By the energy estimate, the convergence rate is at least $p+\frac{1}{2}$ if the semi-discretization is stable. In order to derive a sharper estimate, we follow the same approach as for the half line problem in \S \ref{1dhalfline}. In the remaining part of this section we will only consider the effect of the interface truncation error. The interior truncation error results in an error $\mathcal{O}(h^{2p})$ in the solution. Denote

\begin{equation*}
 \epsilon_{-j}^L(t)=U(x_{-j},t)-u_{-j}^L(t), \quad \epsilon_j^R(t)=U(x_{j},t)-u_{j}^R(t),\quad j=0,1,2,\cdots,
\end{equation*} 
and $\tau=\tau_L=-\tau_R$. The error equation reads
\begin{equation*}
\begin{split}
\epsilon_{tt}^L&=D_L\epsilon^L+\frac{1}{2}H_L^{-1}S_L^Te_{0L}\{\epsilon\}-\frac{1}{2}H_L^{-1}e_{0L}\{S\epsilon\}-\tau H_L^{-1}e_{0L}\{\epsilon\}+h^p_LT^{2p,L,B}, \\
\epsilon_{tt}^R&=D_R\epsilon^R+\frac{1}{2}H_R^{-1}S_R^Te_{0R}\{\epsilon\}-\frac{1}{2}H_R^{-1}e_{0R}\{S\epsilon\}+\tau H_R^{-1}e_{0R}\{\epsilon\}+h^p_RT^{2p,R,B},
\end{split}
\end{equation*}
with the truncation error $h^p_LT^{2p,L,B}$ and $h^p_RT^{2p,R,B}$ as the forcing terms. 
We take the Laplace transform in time of the interface error equation, and obtain
\begin{align}
\tilde{s}_L^2\hat{\epsilon}^L = &\ h_L^2 D_L\hat{\epsilon}^L+\frac{h_L^2}{2}  H_L^{-1}S_L^Te_{0L}\{\hat\epsilon\}-\frac{h_L^2}{2} H_L^{-1}e_{0L}\{S\hat\epsilon\} \label{int3}\\
& -\tau h_L^2 H_L^{-1}e_{0L}\{\hat\epsilon\}+h_L^{p+2}\hat{T}^{2p,L,B}, \nonumber \\
\tilde{s}_R^2\hat{\epsilon}^R = &\ h_R^2D_R\hat{\epsilon}^R+\frac{h_R^2}{2} H_R^{-1}S_R^Te_{0R}\{\hat\epsilon\}-\frac{ h_R^2}{2} H_R^{-1}e_{0R}\{S\hat\epsilon\} \label{int4} \\
&+\tau h_R^2 H_R^{-1}e_{0R}\{\hat\epsilon\}+h_R^{p+2}\hat{T}^{2p,R,B}, \nonumber
\end{align}
where $\tilde{s}_L=sh_L$ and $\tilde{s}_R=sh_R$.
We use normal mode analysis to derive the error estimates for the second and fourth order method. The result is summarized in the following theorem.

\begin{theorem}\label{1dintresult}
For the second and fourth order stable SBP-SAT approximation (\ref{1dintsemi}) of the Cauchy problem with a grid interface, the numerical solution converges to the true solution at rate $q$. The values of $q$ depend on $\tau$, and are listed in Table \ref{1dint_table}.
\begin{table}
\centering
\begin{tabular}{c c c}
\toprule
$2p$ & $q({\tau>\tilde\tau_{2p}})$ & $q(\tau=\tilde\tau_{2p})$ \\ 
\midrule
2 & 2 & 1.5 \\
4 & 4 & 2.5 \\
\bottomrule
\end{tabular}
\caption{Theoretical convergence result for one dimensional wave equation with a grid interface.}
\label{1dint_table}
\end{table}
\end{theorem}

We use the notation $h=h_L=rh_R$ in the proof where $r$ is the mesh size ratio. The proof follows in the same approach as before. We construct the characteristic equation on each side and solve them to obtain the general solution, which are given by (\ref{characteristic_equation}) and  (\ref{characteristic_solution}), respectively. The boundary system is formulated by inserting the general solution to the error equation. The accuracy order is determined by how the solution of this system behaves with respect to $h$ in the vicinity of $\tilde s=0$.

\subsection{Proof of Theorem \ref{1dintresult} for the second order scheme}\label{1dint2o}
Near the interface, the last three rows of  (\ref{int3}) and the first three rows of (\ref{int4}) are affected by the penalty terms. 
The six by six boundary system takes the form
\begin{equation*}
C_{2I}(\tilde{s},\tau)\Sigma=h^3\hat{T}^{2,B}_u,
\end{equation*}
where
\begin{equation*}
\begin{split}
&\Sigma=[\sigma_{1L}, \sigma_{1R}, \hat{\epsilon}_0^L, \hat{\epsilon}_{-1}^L, \hat{\epsilon}_0^R, \hat{\epsilon}_1^R]^T, \\
&\hat{T}_u^{2,B}=[\frac{2}{3}+\frac{1}{3r^2}, 0, 0, -\frac{2}{3r}-\frac{r}{3}, 0, 0]^T\hat{U}_{xxx}(0,s).
\end{split}
\end{equation*}
$C_{2I}(0,\tau)$ is singular with rank 5 if $\tau>\tilde{\tau}_2$, and singular with rank 4 if $\tau=\tilde{\tau}_2$.
The Taylor expansion of $C_{2I}(\tilde{s},\tau)$ at $\tilde{s}=0$ is given in Appendix 3.

In the case $\tau>\tilde{\tau}_2$, we use the Taylor expansion of $C_{2I}(\tilde{s},\tau)$ at $\tilde{s}=0$, and Lemma 3.4 in \cite{Nissen2012} to derive an estimate for $|\Sigma|$.
We need to perform singular value decomposition (SVD) on $C_{2I}(0,\tau)$, which is very difficult due to the variable $\tau$ in the matrix. Instead, we compute numerically the SVD of $C_{2I}(0,\tau)$ for $\tau=1.1\hat{\tau}_2,1.2\hat{\tau}_2,\cdots,5\hat{\tau}_2$. In all cases, we find that $(U^*C'_{2I}(0,\tau)V)_{66}\neq 0$, and that $\hat{T}_u^{2,B}$ is not in the column space of $C_{2I}(0,\tau)$. Therefore, in a vicinity of $|\tilde{s}|=0$, we have $|C_{2I}^{-1}(\tilde s,\tau)|\leq K_{\tilde c}|{\tilde s}|^{-1}$. Thus, 
\begin{equation*}
|\Sigma|\leq  |C_{2I}^{-1}(\tilde{s},\tau)| |\hat{T}_u^{2,B}| h^3 \leq K_c |\tilde{s}|^{-1} |\hat{U}_{xxx}(0,s)| h^3\leq \frac{K_c}{\eta}  |\hat{U}_{xxx}(0,s)| h^2,
\end{equation*}
where $\text{Re}(\tilde{s})\geq h\eta>0$ and $K_c$ is independent of $h$ and $s$. For the error $\hat\epsilon$ in L$_2$ norm $\|\hat\epsilon\|^2_h=\|\hat{\epsilon}^L\|^2_{h_L}+\|\hat{\epsilon}^R\|^2_{h_R}$, we have
\begin{equation*}
\begin{split}
\|\hat\epsilon\|^2_h &= h_L \sum_{i=0}^1 |\hat{\epsilon}^L_{-i}|^2 + h_R \sum_{i=0}^1 |\hat{\epsilon}^R_i|^2 + h_L|\sigma_L|^2\sum_{i=0}^\infty |\kappa_{1L}|^{2i} + h_R|\sigma_R|^2\sum_{i=0}^\infty |\kappa_{1R}|^{2i} \\
&= h_L \sum_{i=0}^1 |\hat{\epsilon}^L_{-i}|^2 + h_R \sum_{i=0}^1 |\hat{\epsilon}^R_i|^2 + \frac{h_L|\sigma_L|^2}{1-|\kappa_{1L}|^2}+ \frac{h_R|\sigma_R|^2}{1-|\kappa_{1R}|^2} \\
&\leq \frac{K_{c1}}{\eta^2}|\hat{U}_{xxx}(0,s)|^2h^5+\frac{K_{c2}}{\eta^3}|\hat{U}_{xxx}(0,s)|^2h^4 \\
&\leq \frac{K_{c12}}{\eta^3}|\hat{U}_{xxx}(0,s)|^2h^4.
\end{split}
\end{equation*}
As before, we obtain
\begin{equation*}
\sqrt{\int_0^{t_f} \|\epsilon\|^2_h dt} \leq h^2\sqrt{\frac{K_{c12}e^{2\eta t_f}}{\eta^3}\int_0^{t_f} |U_{xxx}(0,t)|^2 dt}.
\end{equation*}
Thus, the interface error is $\mathcal{O}(h^2)$ and the accuracy gain is only one order. In this case, the interior error can be estimated to $\mathcal{O}(h^2)$ by an energy estimate, and the overall convergence rate is therefore 2.

In the case $\tau=\tilde{\tau}_2$, the numerical scheme is still stable but the above analysis is no longer valid. By the energy estimate, the convergence rate is $p+\frac{1}{2}=1.5$. 

\subsection{Proof of Theorem \ref{1dintresult} for the fourth order scheme}\label{1dint4o}
In this case, the eight by eight boundary system takes the form
\begin{equation*}
C_{4I}(\tilde{s},\tau)\Sigma=h^4\hat{T}^{4,B}_u,
\end{equation*}
where
\begin{equation*}
\begin{split}
&\Sigma=[\sigma_{1L}, \sigma_{2L}, \sigma_{1R}, \sigma_{2R}, \hat{\epsilon}_0^L, \hat{\epsilon}_{-1}^L, \hat{\epsilon}_0^R, \hat{\epsilon}_1^R]^T, \\
&\hat{T}_u^{4,B}=[\frac{115}{204}-\frac{6}{17r^3}, -\frac{1}{12}, \frac{5}{516}, \frac{11}{588}, \frac{115}{204r^4}-\frac{6}{17r}, -\frac{1}{12r^4}, \frac{5}{516r^4}, \frac{11}{588r^4}]^T\hat{U}_{xxxx}(0,s).
\end{split}
\end{equation*}
The boundary system is also singular in this case. The Taylor expansion of $C_{4I}(\tilde{s},\tau)$ at $\tilde{s}=0$ is given in Appendix 3.
We find that $C_{4I}(0,\tau)$ is singular with rank 7 if $\tau>\tilde{\tau}_4$, and singular with rank 6 if $\tau=\tilde{\tau}_4$.

In the case $\tau>\tilde{\tau}_4$, Lemma 3.4 in \cite{Nissen2012} is applicable since $\hat{T}_u^{4,B}$ is in the column space of $C_{4I}(0,\tau)$. We have
\begin{equation*}
|\Sigma|\leq K_\Sigma |\hat{U}_{xxxx}(0,s)|h^4,
\end{equation*}
where $K_\Sigma$ is independent of $h$ and $s$. As for the second order case, it then follows that the interface error is $\mathcal{O}(h^4)$, that is, the gain in accuracy order at the interface is 2 . In this case, the interior error is $\mathcal{O}(h^4)$ given by an energy estimate. Therefore, the convergence rate is 4.

In the case $\tau=\tilde{\tau}_4$, the numerical scheme is still stable, and the energy estimate gives a sharp estimate. The convergence rate is $p+\frac{1}{2}=2.5$. 

\section{Two dimensional wave equation}\label{2danalysis}
In this section, we consider the second order wave equation in two dimensions:
\begin{equation}
\begin{split}
&U_{tt}=U_{xx}+U_{yy}+F,\ x\in [0,\infty),\ y\in (-\infty,\infty),\ t\geq 0,\\
&U(x,y,0)=f^1(x,y),\ U_t(x,y,0)=f^2(x,y).
\end{split}
\end{equation}
We consider the half plane problem in $x$ with Dirichlet boundary condition
\begin{equation}
U(0,y,t)=g(y,t).
\end{equation} 
The domain in $y$ is the whole real line. For convenience, we require that all data are $2\pi$--periodic in $y$ so that we only need to deal with a finite interval $y\in [0,2\pi]$.
In addition, we assume that the functions $F$, $f^1$, $f^2$ and $g$ are compatible smooth functions with compact support. The L$_2$ norm of the solution is bounded, i.e. $\|U\|<\infty$. Similar to the wave equation in one dimension, with homogeneous Dirichlet boundary condition the continuous energy $E=\|U_t\|^2+\|U_x\|^2+\|U_y\|^2$ is bounded by the data.

Next, we introduce the grid and grid functions. In order to simplify notation, we assume that the mesh size is $h$ in both $x$ and $y$ direction. The grid is
\begin{equation*}
x_i=ih,\ i=0,1,2,\cdots,\quad y_j=jh,\ j=0,1,2,\cdots,N_y,
\end{equation*}
such that $h=\frac{2\pi}{N_y}$, and the grid functions are $u_{ij}(t)\approx U(x_i,y_j,t)$. We arrange the grid functions $u_{ij}(t)$ column wise in a vector $\bm{u}$.   The general notation used to expand an operator from one dimension to two dimensions is $\bm{A_x}=A_x\otimes I_y$, where $\otimes$ is Kronecker product and $I_y$ is the identity operator. With homogeneous boundary condition, the semi-discretized equation reads 
\begin{equation}\label{semi2d}
\bm{u_{tt}}=\bm{D_{xx}u}+\bm{D_{yyp}u}-\bm{H^{-1}S^TE_0u}-\frac{\tau}{h}\bm{H^{-1}E_0u}+\bm{F_g}.
\end{equation}
$D_{xx}$ is the standard SBP operator approximating $\frac{\partial^2}{\partial x^2}$ and $D_{yyp}$ is the standard central finite difference operator approximating  $\frac{\partial^2}{\partial y^2}$. In the same way as for the wave equation in one dimension, a discrete energy estimate is obtained if the penalty parameter $\tau$ is chosen so that $\tau\geq\frac{1}{\alpha_{2p}}$, where the values of $\alpha_{2p}$ are listed in Table \ref{alphavalue}. 

\subsection{Accuracy analysis}
We analyze the accuracy of the numerical scheme by normal mode analysis. Denote $\bm{\epsilon}$ the pointwise error due to the boundary truncation error. Then the boundary error equation is 
\begin{equation*}
\bm{\epsilon_{tt}}=\bm{D_{xx}\epsilon}+\bm{D_{yyp}\epsilon}-\bm{H^{-1}S^TE_0\epsilon}-\frac{\tau}{h}\bm{H^{-1}E_0\epsilon}+h^p\bm{T^{2p}},
\end{equation*}
where $\|\epsilon\|_h<\infty$. We take the Laplace transform in $t$ and Fourier transform in $y$, and obtain for every $\omega$
\begin{equation}\label{err2dfl}
{\tilde{s}^2_+}\tilde{\hat{\epsilon}}_\omega=h^2D_{xx}\tilde{\hat{\epsilon}}_\omega-h^2H^{-1}S^TE_0\tilde{\hat{\epsilon}}_\omega-\tau hH^{-1}E_0\tilde{\hat{\epsilon}}_\omega+h^{p+2}\tilde{\hat{T}}^{2p}_\omega,
\end{equation}
where $\tilde{s}_+=\sqrt{\tilde{s}^2-h^2\tilde{D}_{yyp}^\omega}$, $\tilde{s}=sh$ and $\tilde{D}_{yyp}^\omega$ is the Fourier transform of $D_{yyp}$. For different orders of accuracy, we have
\begin{equation*}
\tilde{D}_{yyp}^\omega=\begin{cases}
-\frac{4}{h^2}\sin^2\frac{\omega h}{2} & \text{if } 2p=2, \\
-\frac{4}{h^2}\sin^2\frac{\omega h}{2}(1+\frac{1}{3}\sin^2\frac{\omega h}{2})& \text{if } 2p=4, \\
-\frac{4}{h^2}\sin^2\frac{\omega h}{2}(1+\frac{1}{3}\sin^2\frac{\omega h}{2}+\frac{8}{45}\sin^4\frac{\omega h}{2}) & \text{if } 2p=6.
\end{cases}
\end{equation*}
Note that $\tilde{D}_{yyp}^\omega$ is nonpositive. 
\begin{lemma}\label{splus}
For Re$(\tilde{s})>0$ and $\tilde{D}_{yyp}^\omega$ defined above, it follows that Re$(\tilde{s}_+)>0$.
\end{lemma}
\begin{proof}
By the definition of principal square root, Re$(\tilde{s}_+)\geq 0$. Assume Re$(\tilde{s}_+)= 0$, then $\tilde{s}^2-h^2\tilde{D}_{yyp}^\omega$ is real and nonpositive, i.e. Re$(\tilde{s}^2-h^2\tilde{D}_{yyp}^\omega)\leq 0$. We then obtain Re$(\tilde{s}^2)\leq$ Re$(h^2\tilde{D}_{yyp}^\omega)\leq 0$. This contradicts Re$(\tilde s)>0$. Therefore, Re$(\tilde{s}_+)>0$.
\end{proof}

The truncation error $\bm{{T}^{2p}}$ can be written as $\bm{{T}^{2p}}=[T^{2p}_0;T^{2p}_1;\cdots]$, where $T^{2p}_i=[T^{2p}_{i0}; T^{2p}_{i1};\cdots;T^{2p}_{iN_y} ]$ is an $N_y+1$-by-1 vector. For $2p^{th}$ order accurate method, we have 
\begin{equation*}
\begin{split}
&T^{2p}_{kj}=c_kU^{(p+2)}(x_k,y_j,t),\quad 0\leq k\leq m-1,\quad 0\leq j\leq N_y, \\
&T^{2p}_{k}=\bm{0},\quad k\geq m
\end{split}
\end{equation*}
to the leading order, where $m=1,4,6,8,11$ for $2p=2,4,6,8,10$, $c_k$ is a constant independent of $h$, $U^{(p+2)}:=\frac{\partial^{p+2} U}{\partial x^{p+2}}$ and $\bm{0}$ is a zero vector of size $N_y+1$-by-$1$ . We take the Fourier transform in $y$ and Laplace transform in $t$, and obtain
\begin{equation*}
\begin{split}
&\tilde{\hat{T}}^{2p}_{k}=c_k\tilde{\hat{U}}^{(p+2)}(x_k,\omega,s),\quad 0\leq k\leq m-1, \\
&\tilde{\hat{T}}^{2p}_{k}=\bm{0},\quad k\geq m.
\end{split}
\end{equation*}
Similar to the one dimensional case, we only need to consider a vicinity of $|\tilde s|=0$ with Re$(\tilde s)>0$. By Lemma \ref{splus}, we have Re$(\tilde s_+)>0$. The boundary part of Equation (\ref{err2dfl}) can be written in the matrix vector multiplication form $C(\tilde s_+,\tau)\Sigma_\omega=h^{p+2}\tilde{\hat{T}}^{2p,B}_\omega$. 

We then need to investigate the invertibility of $C(\tilde s_+,\tau)$ and the dependence of $|C^{-1}(\tilde s_+,\tau)\tilde{\hat{T}}^{2p}_\omega|$ on $h$. We note that Equation (\ref{err2dfl}) is analogous to the error equation (\ref{Berreqn1dLaplaceTilde}) in one dimensional case, and the matrix--valued function $C(\cdot,\tau)$ is the same with only the argument changed. Therefore, based on the normal mode analysis for the one dimensional problem in \S\ref{1d2o}-\ref{1d6o}, for $2p=2,4,6$ and $\tau>\tau_{2p}$ the boundary error in the Fourier and Laplace space is bounded as
\begin{equation*} 
\|\tilde{\hat\epsilon}_\omega\|^2_h\leq K_\omega \left|\frac{d^{p+2}}{dx^{p+2}}\tilde{\hat U}(0,\omega,s)\right|^2 h^{2p+4},
\end{equation*}
where $K_\omega$ depends on $\eta$ but not $h$. Since the above estimate holds for every $\omega$, we sum them up and obtain
\begin{equation*} 
\sum_\omega \|\tilde{\hat\epsilon}_\omega\|^2_h\leq \sum_\omega K_\omega \left|\frac{d^{p+2}}{dx^{p+2}}\tilde{\hat U}(0,\omega,s)\right|^2h^{2p+4}.
\end{equation*}
By using the Parseval's relation in the Fourier space, we obtain
\begin{equation*} 
\|\hat\epsilon\|^2_h\leq K_L \int_0^{2\pi} \left|\frac{d^{p+2}}{dx^{p+2}}\hat U(0,y,s)\right|^2 dy\ h^{2p+4}.
\end{equation*}
We then use Parseval's relation in the Laplace space, and obtain
\begin{equation*} 
\sqrt{\int_0^{t_f}\|\epsilon\|^2_h dt}\leq h^{p+2}\sqrt{K_L e^{2\eta t_f} \int_0^{t_f} \int_0^{2\pi} \left|\frac{d^{p+2}}{dx^{p+2}} U(0,y,t)\right|^2 dy dt}.
\end{equation*}
For $2p=2,4,6$, the boundary error is bounded in $\mathcal{O}(h^{p+2})$. We know from the energy estimate that the interior error is bounded in $\mathcal{O}(h^{2p})$. Therefore, we conclude that for second, fourth and sixth order methods, the convergence rates are 2, 4 and 5 only if $\tau>\tau_{2p}$. If $\tau=\tau_{2p}$, the boundary system is singular. The convergence rate $p+\frac{1}{2}$ given by the energy estimate is sharp. The numerical experiments in \S\ref{numericalexperiments} agree with this conclusion.

The above analysis can be carried out also for the Neumann problem and the interface problem (with conforming grid interfaces) in two dimensions. Comparing with the corresponding one dimensional problem, the only difference in the error equation is that $\tilde s$ is replaced by $\tilde s_+$. By Lemma \ref{splus} and the analysis thereafter, we conclude that the convergence result for the Neumann problem and the interface problem in two dimensions is the same as that for the corresponding one dimensional problem. For numerical experiments of these cases, we refer to \cite{Virta2014}.

It is not straightforward to generalize the above analysis to a two dimensional problem with non--periodic boundary conditions at all boundaries, or to a two dimensional problem with a non--conforming grid interface. Many numerical experiments of these cases, however, show an agreement with the convergence rate that is conjectured from the analysis of the corresponding one dimensional problem.

\section{Numerical experiments}\label{numericalexperiments}
In this section, we perform numerical experiments to verify the accuracy analysis. We investigate how the accuracy of the numerical solution is affected by the large truncation error localized near the boundary and grid interface. In the analysis, we use the half line and half plane problem. However, in the numerical experiments, we use a bounded domain and impose boundary conditions on all boundaries weakly.
We discretize in time using the classical fourth order Runge--Kutta method. The step size $\Delta t$ in time is chosen so that the temporal error has a negligible impact of the spatial convergence result. The value of $\Delta t$ is given in each numerical experiment.

The L$_2$ error and maximum error are computed as the norm of the difference between the exact solution $u^{ex}$ and the numerical solution $u^h$ according to
\begin{equation*}
\begin{split}
&\|u^h-u^{ex}\|_{\text{L}_2}=\sqrt{h^d(u^h-u^{ex})^T(u^h-u^{ex})}, \\
&\|u^h-u^{ex}\|_{\infty}=\max |u^h-u^{ex}|, 
\end{split}
\end{equation*}
where $d$ is the dimension of the equation. The convergence rates are computed by
\begin{equation*}
q=\log\left(\frac{\|u^h-u^{ex}\|}{\|u^{2h}-u^{ex}\|}\right) \bigg/ \log\left(\frac{1}{2}\right).
\end{equation*}

\subsection{One dimensional wave equation}
We choose the analytical solution
\begin{equation}\label{1danalytic}
u=\cos(10\pi x+1)\cos(10\pi t+2),\quad 0\leq x\leq 1,\ 0\leq t\leq 2,
\end{equation}
and use it to impose the Dirichlet boundary conditions at $x=0$ and $x=1$. The analytical solution and its derivatives do not vanish at the boundaries at the final time $t=2$. The L$_2$ error, the convergence rates in L$_2$ norm and maximum norm are shown in Table \ref{1d_no_int}. The convergence rates behave as expected. With a $2p^{th} (2p=2,4,6)$ order accuracy method, the convergence rate is $p+\frac{1}{2}$ in L$_2$ norm if the penalty parameter $\tau$ equals its limit. In this case, the pointwise error in the Laplace space is bounded in $\mathcal{O}(h^p)$. We can hope for $p^{th}$ order convergence rate in the maximum norm, which is observed as well. With an increased penalty parameter, the accuracy of the numerical solution is improved significantly, and the convergence rate $\min(2p,p+2)$ is obtained. For second and fourth order accuracy cases, the accuracy difference between $\tau=1.2\tau_{2p}$ and $\tau=3\tau_{2p}$ is almost invisible. For sixth order accuracy case, one can see the small improvement in accuracy between $\tau=1.2\tau_{2p}$ and $\tau=3\tau_{2p}$. The step size in time is $\Delta t=0.1h$ in all cases, except for the last two grid refinements of the sixth order scheme with $\tau=3\tau_{2p}$ where $\Delta t=0.05h$ is used in order to observe the expected convergence behaviour. Note that these step sizes are much smaller than the step size restricted by the stability condition. With $\tau=1.2\tau_{2p}$, the numerical method is stable if $\Delta t/h\leq c$, where the Courant number $c=1.3,0.9$ and 0.6 for second, fourth and sixth order method. For an increased penalty parameter $\tau=3\tau_{2p}$, the corresponding Courant numbers are 0.8, 0.5 and 0.4. 
\begin{table}
\footnotesize 
\centering
\begin{tabular}{c c c c c c c}
\toprule
$\tau=\tau_{2p}\quad$          &         Second Order            &    &  Fourth Order  &   &   Sixth Order &  \\ \midrule
$N$ &  $\|u^h-u^{ex}\|_{\text{L}_2}$ & $q_{\text{L}_2}/q_\infty$ & $\|u^h-u^{ex}\|_{\text{L}_2}$ & $q_{\text{L}_2}/q_\infty$  & $\|u^h-u^{ex}\|_{\text{L}_2}$ & $q_{\text{L}_2}/q_\infty$  \\ 
51   &  4.35                        &         & 1.60                         &         & 1.10                         & \\
101 &  1.58                        & 1.46/1.07 & $2.53\cdot 10^{-1}$ & 2.66/2.21 & $1.37\cdot 10^{-1}$ & 3.00/2.54 \\
201 & $5.62\cdot 10^{-1}$ & 1.49/1.05 & $4.27\cdot 10^{-2}$ & 2.56/2.21 & $1.34\cdot 10^{-2}$ & 3.35/3.02 \\
401 & $1.99\cdot 10^{-1}$ & 1.50/1.03 & $7.46\cdot 10^{-3}$ & 2.52/2.14 & $1.22\cdot 10^{-3}$ & 3.46/3.07 \\
801 & $7.04\cdot 10^{-2}$ & 1.50/1.02 & $1.31\cdot 10^{-3}$ & 2.50/2.08 & $1.08\cdot 10^{-4}$ & 3.49/3.05 \\
\midrule
 $\tau=1.2\tau_{2p}\quad$          &         Second Order            &    &  Fourth Order  &   &   Sixth Order &  \\ 
$N$ &  $\|u^h-u^{ex}\|_{\text{L}_2}$ & $q_{\text{L}_2}/q_\infty$ & $\|u^h-u^{ex}\|_{\text{L}_2}$ & $q_{\text{L}_2}/q_\infty$  & $\|u^h-u^{ex}\|_{\text{L}_2}$ & $q_{\text{L}_2}/q_\infty$  \\ 
51   &  $7.63\cdot 10^{-1}$ &                 & $4.84\cdot 10^{-2}$ &                 & $1.04\cdot 10^{-2}$ & \\
101 &  $1.51\cdot 10^{-1}$ & 2.34/2.42 & $1.74\cdot 10^{-3}$ & 4.80/4.85 & $8.86\cdot 10^{-5}$ & 6.87/6.90 \\
201 &  $3.56\cdot 10^{-2}$ & 2.09/2.12 & $1.08\cdot 10^{-4}$ & 4.01/4.12 & $2.51\cdot 10^{-6}$ & 5.14/4.46 \\
401 &  $8.77\cdot 10^{-3}$ & 2.02/2.03 & $6.93\cdot 10^{-6}$ & 3.96/3.93 & $7.28\cdot 10^{-8}$ & 5.11/4.38 \\
801 &  $2.19\cdot 10^{-3}$ & 2.00/2.01 & $4.41\cdot 10^{-7}$ & 3.97/4.00 & $1.54\cdot 10^{-9}$ & 5.56/5.13 \\
\midrule
 $\tau=3\tau_{2p}\quad$          &         Second Order            &    &  Fourth Order  &   &   Sixth Order &  \\ 
$N$ &  $\|u^h-u^{ex}\|_{\text{L}_2}$ & $q_{\text{L}_2}/q_\infty$ & $\|u^h-u^{ex}\|_{\text{L}_2}$ & $q_{\text{L}_2}/q_\infty$  & $\|u^h-u^{ex}\|_{\text{L}_2}$ & $q_{\text{L}_2}/q_\infty$  \\ 
51   &  $6.03\cdot 10^{-1}$ &                 & $3.50\cdot 10^{-2}$ &                 & $7.14\cdot 10^{-3}$   & \\
101 &  $1.47\cdot 10^{-1}$ & 2.03/2.11 & $1.70\cdot 10^{-3}$ & 4.37/4.32 & $6.87\cdot 10^{-5}$   & 6.70/6.22 \\
201 &  $3.55\cdot 10^{-2}$ & 2.05/2.10 & $1.07\cdot 10^{-4}$ & 3.98/4.00 & $1.37\cdot 10^{-6}$   & 5.65/5.47 \\
401 &  $8.77\cdot 10^{-3}$ & 2.02/2.03 & $6.92\cdot 10^{-6}$ & 3.95/3.96 & $1.84\cdot 10^{-8}$   & 6.22/5.06 \\
801 &  $2.19\cdot 10^{-3}$ & 2.00/2.01 & $4.41\cdot 10^{-7}$ & 3.97/3.98 & $3.44\cdot 10^{-10}$ & 5.74/4.81 \\
\bottomrule
\end{tabular}
\caption{Convergence study for one dimensional wave equation with Dirichlet boundary conditions.}
\label{1d_no_int}
\end{table}

\begin{table}
\footnotesize 
\centering
\begin{tabular}{c c c  c c c c c c}
\toprule
        & Fourth Order  &   &  Perturbed &  &  Sixth Order  &   &  Perturbed &   \\ \midrule
$N$ & $\|u^h-u^{ex}\|_{\text{L}_2}$ & $q_{\text{L}_2}$ & $\|u^h-u^{ex}\|_{\text{L}_2}$ & $q_{\text{L}_2}$  & $\|u^h-u^{ex}\|_{\text{L}_2}$ & $q_{\text{L}_2}$ & $\|u^h-u^{ex}\|_{\text{L}_2}$ & $q_{\text{L}_2}$\\ \midrule
51    & $1.84\cdot 10^{-2}$ &          & $1.46\cdot 10^{-2}$ &        & $9.56\cdot 10^{-3}$ &         & $9.56\cdot 10^{-3}$ & \\
101  & $1.36\cdot 10^{-3}$ & 3.76  & $3.79\cdot 10^{-3}$ &1.95 & $1.94\cdot 10^{-5}$ & 8.94 & $7.40\cdot 10^{-5}$ & 7.01\\
201  & $8.17\cdot 10^{-5}$ & 4.06  & $6.62\cdot 10^{-4}$ &2.52 & $3.04\cdot 10^{-7}$ & 6.00 & $7.40\cdot 10^{-6}$ & 3.32\\
401  & $4.88\cdot 10^{-6}$ & 4.07  & $9.03\cdot 10^{-5}$ &2.88 & $8.64\cdot 10^{-9}$ & 5.14& $5.12\cdot 10^{-7}$ & 3.85\\
801  & $2.97\cdot 10^{-7}$ & 4.04  & $1.16\cdot 10^{-5}$ &2.96 & $1.20\cdot 10^{-10}$ &6.18 & $3.29\cdot 10^{-8}$ & 3.96\\ \bottomrule
\end{tabular}
\caption{Convergence study for one dimensional wave equation with Neumann boundary condition.}
\label{1d_neu}
\end{table}

To test the Neumann problem, we use (\ref{1danalytic}) as the analytical solution and impose Neumann boundary condition at two boundaries $x=0$ and $x=1$. The L$_2$ error and the convergence rate in L$_2$ norm are shown in Table \ref{1d_neu} in Column 2 and 3 for the fourth order scheme. Fourth order convergence rate is clearly observed, which corresponds to the optimal accuracy order gain of two from the boundary truncation error. In the accuracy analysis in \S\ref{section_Neumann}, the optimal convergence rate relies on the fact that the truncation error vector is in the column space of $C_{4N}(0)$. We add a dissipative term to the boundary block of the SBP operator so that the boundary truncation error vector is in the same magnitude but no longer in the column space of $C_{4N}(0)$. The result is listed in Column 4 and 5 in Table $\ref{1d_neu}$. Now the gain in accuracy order is only one, and third order convergence rate is obtained as expected. Note that an energy estimate is also valid for the perturbed problem and guarantees stability. We repeat the same experiments with the sixth order scheme. The convergence rate is between five and six. We then add a dissipative term in the same way as for the fourth order scheme. The optimal accuracy order gain is lost, and the convergence rate drops to four.

Next, we test how the accuracy and convergence are affected by the large truncation error localized near a grid interface. We choose the same analytical solution (\ref{1danalytic}). The grid interface is located at $x=0.5$ in the computational domain $\Omega=[0,1]$, and the grid size ratio is $2:1$. The step size in time is $\Delta t=0.1h$, where $h$ is the smaller grid size.
Note that the analytical solution does not vanish at the grid interface. We use the SAT technique to impose the outer boundary condition weakly and choose the boundary penalty parameters strictly larger than its limit. The interface conditions are also imposed by the SAT technique. We use different interface penalty parameters to see how they affect the accuracy and convergence. The results are shown in Table \ref{1d_int}, where $N$ denotes the number of grid points in the coarse domain. 

\begin{table}
\footnotesize 
\centering
\begin{tabular}{c c c c c c c}
\toprule
$\tau=\tau_{2p}\quad$          &         Second Order            &    &  Fourth Order  &   &   Sixth Order &  \\ \midrule
$N$ &  $\|u^h-u^{ex}\|_{\text{L}_2}$ & $q_{\text{L}_2}/q_\infty$ & $\|u^h-u^{ex}\|_{\text{L}_2}$ & $q_{\text{L}_2}/q_\infty$  & $\|u^h-u^{ex}\|_{\text{L}_2}$ & $q_{\text{L}_2}/q_\infty$  \\ 
26   &  2.68                         &                 & 1.26                         &                 & $1.76\cdot 10^{-1}$ & \\
51   &  1.06                         & 1.34/0.84 & $1.86\cdot 10^{-1}$ & 2.76/2.26 & $4.56\cdot 10^{-2}$ & 1.94/1.45\\
101 &  $3.91\cdot 10^{-1}$ & 1.43/0.93 & $2.76\cdot 10^{-2}$ & 2.75/2.25 & $6.58\cdot 10^{-3}$ & 2.79/2.29 \\
201 &  $1.41\cdot 10^{-1}$ & 1.47/0.97 & $4.33\cdot 10^{-3}$ & 2.67/2.17 & $6.76\cdot 10^{-4}$ & 3.28/2.78 \\
401 &  $5.03\cdot 10^{-2}$ & 1.49/0.99 & $7.15\cdot 10^{-4}$ & 2.60/2.10 & $6.34\cdot 10^{-5}$ & 3.41/2.91 \\
\midrule
 $\tau=1.2\tau_{2p}\quad$          &         Second Order            &    &  Fourth Order  &   &   Sixth Order &  \\ 
$N$ &  $\|u^h-u^{ex}\|_{\text{L}_2}$ & $q_{\text{L}_2}/q_\infty$ & $\|u^h-u^{ex}\|_{\text{L}_2}$ & $q_{\text{L}_2}/q_\infty$  & $\|u^h-u^{ex}\|_{\text{L}_2}$ & $q_{\text{L}_2}/q_\infty$  \\ 
26   &  $5.49\cdot 10^{-1}$ &                 & $2.46\cdot 10^{-2}$ &                 & $1.09\cdot 10^{-2}$ & \\
51   &  $9.33\cdot 10^{-2}$ & 2.56/2.57 & $8.41\cdot 10^{-4}$ & 4.87/4.83 & $9.32\cdot 10^{-5}$ & 6.87/6.69\\
101 &  $2.21\cdot 10^{-2}$ & 2.08/2.11 & $5.36\cdot 10^{-5}$ & 3.97/4.02 & $1.84\cdot 10^{-6}$ & 5.66/5.11 \\
201 &  $5.48\cdot 10^{-3}$ & 2.01/2.03 & $3.56\cdot 10^{-6}$ & 3.91/3.90 & $4.48\cdot 10^{-8}$ & 5.36/4.41 \\
401 &  $1.37\cdot 10^{-3}$ & 2.00/2.01 & $2.30\cdot 10^{-7}$ & 3.95/3.94 & $1.05\cdot 10^{-9}$ & 5.41/5.24 \\
\midrule
 $\tau=3\tau_{2p}\quad$          &         Second Order            &    &  Fourth Order  &   &   Sixth Order &  \\ 
$N$ &  $\|u^h-u^{ex}\|_{\text{L}_2}$ & $q_{\text{L}_2}/q_\infty$ & $\|u^h-u^{ex}\|_{\text{L}_2}$ & $q_{\text{L}_2}/q_\infty$  & $\|u^h-u^{ex}\|_{\text{L}_2}$ & $q_{\text{L}_2}/q_\infty$  \\ 
26   &  $3.83\cdot 10^{-1}$ &                 & $2.07\cdot 10^{-2}$ &                 & $7.96\cdot 10^{-3}$ & \\
51   &  $9.04\cdot 10^{-2}$ & 2.08/2.12 & $8.31\cdot 10^{-4}$ & 4.64/4.64 & $5.69\cdot 10^{-5}$ & 7.13/6.42\\
101 &  $2.20\cdot 10^{-2}$ & 2.04/2.07 & $5.33\cdot 10^{-5}$ & 3.96/3.97 & $1.17\cdot 10^{-6}$ & 5.61/5.49 \\
201 &  $5.47\cdot 10^{-3}$ & 2.01/2.02 & $3.55\cdot 10^{-6}$ & 3.91/3.90 & $1.53\cdot 10^{-8}$ & 6.25/5.10 \\
401 &  $1.37\cdot 10^{-3}$ & 2.00/2.00 & $2.30\cdot 10^{-7}$ & 3.95/3.96 & $2.97\cdot 10^{-10}$ & 5.68/4.78 \\
\bottomrule
\end{tabular}
\caption{Convergence study for one dimensional wave equation with a grid interface.}
\label{1d_int}
\end{table}

According to the accuracy analysis in \S\ref{section_interface}, for second and fourth order accuracy methods with $\tau=\tau_{2p}$, the expected convergence rates are 1.5 and 2.5 in L$_2$ norm. This is clearly observed in Table \ref{1d_int}. We also observe in Table \ref{1d_int} that the convergence rate in the maximum norm behaves as $p$. With an increased penalty parameter, a much better convergence result is obtained. For the second and fourth order methods, we get the optimal (second and fourth, respectively) order of convergence in both L$_2$ norm and maximum norm. For the sixth order scheme, we expect fifth order convergence in L$_2$ norm. The numerical experiments actually give a higher convergence result. This has also been observed for Schr\"{o}dinger equation in \cite{Nissen2012,Nissen2013}. 
\begin{remark}
The numerical experiments with $8^{th}$ and $10^{th}$ order SBP operators are also carried out. When the penalty parameter equals its lower limit, 4.5 and 5.5 convergence rate in L$_2$ norm, and 4 and 5 convergence in maximum norm are clearly observed. With a larger penalty parameter, we expect the convergence to be 6 and 7, respectively. However, due to the high accuracy, the error decreases fast to the machine precision. Therefore, the expected convergence rates are not clearly observed.  
\end{remark}

\subsection{Two dimensional wave equation}
For two dimensional simulation, we choose the analytical solution
\begin{equation*}
u=\cos(12x+1)\cos(4\pi y+2)\cos(\sqrt{12^2+(4\pi)^2}t+3),\quad 0\leq x\leq 1,\ 0\leq y\leq 1,\ 0\leq t\leq 2,
\end{equation*}
and use it to impose the Dirichlet boundary condition at $x=0$ and $x=1$. Periodic boundary condition is imposed at $y=0$ and $y=1$. The numbers of grid points are the same in $x$ and $y$ direction, and denoted by $N$ in Table \ref{2d_no_int}. The step size in time is $\Delta t=0.1h$. We clearly observe that with a $2p^{th}$ order method and $\tau=\tau_{2p}$ the convergence rates in L$_2$ norm are $p+\frac{1}{2}$ and in maximum norm are $p$. With an increased penalty parameter, the errors in L$_2$ norm are much smaller. The convergence rates are $\min(2p,p+2)$ in both L$_2$ norm and maximum norm. The results from the numerical experiments agree well with the analysis in \S\ref{2danalysis}. 

\begin{table}
\footnotesize 
\centering
\begin{tabular}{c c c c c c c}
\toprule
$\tau=\tau_{2p}\quad$          &         Second Order            &    &  Fourth Order  &   &   Sixth Order &  \\ \midrule
$N$ &  $\|u^h-u^{ex}\|_{\text{L}_2}$ & $q_{\text{L}_2}/q_\infty$ & $\|u^h-u^{ex}\|_{\text{L}_2}$ & $q_{\text{L}_2}/q_\infty$  & $\|u^h-u^{ex}\|_{\text{L}_2}$ & $q_{\text{L}_2}/q_\infty$  \\ 
26   & $1.05\cdot 10^{-1}$ &                 & $3.48\cdot 10^{-2}$ &                 & $2.27\cdot 10^{-2}$ & \\
51   & $3.72\cdot 10^{-2}$ & 1.50/0.95 & $6.48\cdot 10^{-3}$ & 2.43/1.98 & $2.01\cdot 10^{-3}$ & 3.49/2.80 \\
101 & $1.26\cdot 10^{-2}$ & 1.56/1.02 & $1.17\cdot 10^{-3}$ & 2.47/2.03 & $1.74\cdot 10^{-4}$ & 3.53/3.05 \\
201 & $4.34\cdot 10^{-3}$ & 1.54/1.02 & $2.07\cdot 10^{-4}$ & 2.49/2.03 & $1.52\cdot 10^{-5}$ & 3.52/3.06 \\
401 & $1.51\cdot 10^{-3}$ & 1.53/1.01 & $3.68\cdot 10^{-5}$ & 2.50/2.02 & $1.34\cdot 10^{-6}$ & 3.51/3.04 \\
\midrule
 $\tau=1.2\tau_{2p}\quad$          &         Second Order            &    &  Fourth Order  &   &   Sixth Order &  \\ 
$N$ &  $\|u^h-u^{ex}\|_{\text{L}_2}$ & $q_{\text{L}_2}/q_\infty$ & $\|u^h-u^{ex}\|_{\text{L}_2}$ & $q_{\text{L}_2}/q_\infty$  & $\|u^h-u^{ex}\|_{\text{L}_2}$ & $q_{\text{L}_2}/q_\infty$  \\ 
26   &  $6.74\cdot 10^{-2}$ &                 & $9.52\cdot 10^{-3}$ &                 & $3.20\cdot 10^{-3}$ & \\
51   &  $1.77\cdot 10^{-2}$ & 1.93/1.96 & $2.13\cdot 10^{-4}$ & 5.49/5.22 & $1.06\cdot 10^{-4}$ & 4.91/4.26 \\
101 &  $4.53\cdot 10^{-3}$ & 1.96/1.96 & $1.62\cdot 10^{-5}$ & 3.71/3.57 & $1.26\cdot 10^{-6}$ & 6.40/6.04 \\
201 &  $1.14\cdot 10^{-3}$ & 1.99/2.01 & $6.62\cdot 10^{-7}$ & 4.62/4.21 & $3.54\cdot 10^{-8}$ & 5.15/4.79 \\
401 &  $2.86\cdot 10^{-4}$ & 2.00/2.00 & $3.83\cdot 10^{-8}$ & 4.11/4.01 & $6.00\cdot 10^{-10}$ & 5.88/5.53 \\
\bottomrule
\end{tabular}
\caption{Convergence study for two dimensional wave equation.}
\label{2d_no_int}
\end{table}

\section{Conclusion}\label{conclusion} 
For the second order wave equation a stable numerical scheme does not automatically satisfy the determinant condition. We have considered stable SBP--SAT schemes for the Dirichlet, Neumann and interface problems. In all these cases the boundary truncation error is larger than the interior truncation error. We find that only the Dirichlet problem with penalty parameter greater than its limit value satisfies the determinant condition, and in agreement with \cite{Gustafsson2013,Svard2006}, and a gain of two in accuracy order for the boundary error follows. Also in this case, the determinant condition is necessary for the optimal convergence rate. 

For all the other cases the determinant condition is not satisfied. For most of these cases there is a gain in accuracy of order one or even two, though the energy estimate only implies a gain of order 1/2. We show that a careful analysis of the effect of the boundary truncation error yields sharp estimates. We also show that the results from analysis in one space dimension are valid in two space dimensions.

In particular we have found that to get the optimal gain in accuracy order for the Dirichlet problem and the grid interface problem, the penalty parameter should be chosen larger than its limit value given by stability analysis. However, we should not choose a very large penalty parameter since it does not decrease further the error in the solution but leads to a very small Courant number, thus an unnecessarily small step size in time.

\section*{Appendix 1}
The matrix $C_{4D}(0,\tau)$ in \S\ref{1d4o} is 
\begin{equation*}
C_{4D}(0,\tau)=\begin{bmatrix}
-3 & 3-4\sqrt{3} & \frac{-122+48\tau}{17} & 5 \\
-1 & -1 & \frac{85}{59} & 2 \\
\frac{55}{43} & \frac{85+12\sqrt{3}}{43} & -\frac{68}{43} & -\frac{59}{43} \\
-\frac{1}{49} & -\frac{37+3\sqrt{3}}{49} & \frac{17}{49} & 0
\end{bmatrix}
\end{equation*}

The first four columns of the matrix $C_{6D}(0,\tau)$ in \S\ref{1d6o} are 
\begin{equation*}
\underbrace{C_{6D}(0,\tau)}_{\text{column 1--4}}=\begin{bmatrix}
k(\tau) &   8.024525606 &  -8.215717879 & 2.979430728\\
                   1.823470407 &    2.35039818 &  -1.867463026 & 0.5704778157\\
                  -4.136345752 &  -4.137556867 &   8.108447068 & -4.615068241 \\
                  0.8614698016 &   1.159078186 &  -3.511693724  & 2.565973129 \\
                 -0.0951377428 & -0.9156510516 &   1.987601033 & -1.25102831 \\
                -0.04002001476 &  0.1875223549 & -0.3471267779  & 0.1996244378
\end{bmatrix}
\end{equation*}
where $k(\tau)=3.165067038\tau - 9.382128117$, and the last two columns of the matrix $C_{6D}(0,\tau)$ in \S\ref{1d6o} are 
\begin{equation*}
\underbrace{C_{6D}(0,\tau)}_{\text{column 5--6}}=\begin{bmatrix}
    3.368616929 + 0.02305363188$i$ & 3.368616929 - 0.02305363188$i$ \\
 1.00246526 + 0.04694800264$i$ &   1.00246526 - 0.04694800264$i$ \\
 -6.789853693 - 0.2265966418$i$ & -6.789853693 + 0.2265966418$i$ \\
4.592688087 + 0.1965568935$i$ &  4.592688087 - 0.1965568935$i$ \\
-3.188300477 - 0.2835125258$i$ & -3.188300477 + 0.2835125258$i$ \\
0.3280801113 + 0.1028755037$i$ & 0.3280801113 - 0.1028755037$i$
\end{bmatrix}
\end{equation*}

\section*{Appendix 2}
In \S\ref{sec-Neumann-46}, the boundary system for fourth and sixth order schemes are analyzed. We have for the fourth order scheme
\begin{equation*}
C_{4N}(0)=\begin{bmatrix}
\frac{54}{17} & -\frac{59}{17} & \frac{5}{17} & \frac{11-4\sqrt{3}}{17} \\
-1 & 2 & -1 & -1 \\
\frac{4}{43} & -\frac{59}{43} & \frac{55}{43} & \frac{85+12\sqrt{3}}{43} \\
\frac{1}{49} & 0 & -\frac{1}{49} & -\frac{37+8\sqrt{3}}{49}
\end{bmatrix},\ C'_{4N}(0)=\begin{bmatrix}
0 & 0 & -0.0588 & 0\\
0 & 0 & 0 & 0 \\
0 & 0 & 1.1860 & 0 \\
0 & 0 & -0.0408 & 0
\end{bmatrix},
\end{equation*}
and 
\begin{equation*}
\hat T_u^{4,Neu,B}=\begin{bmatrix}
\frac{43}{204} \\
-\frac{1}{12} \\
\frac{5}{516} \\
\frac{11}{588}
\end{bmatrix}\hat U_{xxxx}(0,s).
\end{equation*}
For sixth order scheme, we have 
\begin{equation*}
\underbrace{C_{6N}(0)}_{\text{column 1--4}}=\begin{bmatrix}
   3.8056512077 & -4.6357425452  &  1.2794832344 & -0.4493918968 \\
  -1.0534129693 &   2.3503981797 & -1.8674630262 &   0.5704778157 \\
   0.6441780401 & -4.1375568671  &  8.1084470675 & -4.6150682405 \\ 
  -0.2133575916 & 1.1590781862   & -3.5116937240 &  2.5659731293 \\
   0.1790783293 & -0.9156510516  & 1.9876010325  & -1.2510283103 \\
  -0.0400200148 &  0.1875223549  & -0.3471267779 & 0.1996244378 
  \end{bmatrix},
\end{equation*}
\begin{equation*}
\underbrace{C_{6N}(0)}_{\text{column 5--6}}=\begin{bmatrix}
-0.8104124471 - 0.0403172285$i$   & -0.8104124471 + 0.0403172285$i$ \\
  1.0024652602 + 0.0469480026$i$ &  1.0024652602 - 0.0469480026$i$ \\
 -6.7898536930 - 0.2265966418$i$  & -6.7898536930 + 0.2265966418$i$ \\
  4.5926880872 + 0.1965568935$i$ &  4.5926880872 - 0.1965568935$i$ \\
 -3.1883004772 - 0.2835125258$i$  & -3.1883004772 + 0.2835125258$i$ \\
  0.3280801113 + 0.1028755037$i$  & 0.3280801113 - 0.1028755037$i$
\end{bmatrix},
\end{equation*}
\begin{equation*}
C'_{6N}(0)=\begin{bmatrix}
                   0&                   0&                   0&  -0.2598847290&                   0&                   0 \\
                   0&                   0&                   0&   0.3269055745&                   0&                   0 \\ 
                   0&                   0&                   0&  -1.7658674536&                   0&                   0 \\
                   0&                   0&                   0&   1.8336101263&                   0&                   0  \\
                   0&                   0&                   0&  -0.6935339173&                   0&                   0 \\
                   0&                   0&                   0&   0.0921421124&                    0&                  0 \\
\end{bmatrix},
\end{equation*}

\begin{equation*}
\hat T_u^{6,Neu,B}=\begin{bmatrix}
-0.3287481378 \\
0.2200796359 \\
  -0.5608447068 \\
   0.2044006966 \\
  -0.1710063053 \\
   0.0514543047 
\end{bmatrix}\hat U_{xxxxx}(0,s).
\end{equation*}

\section*{Appendix 3}
In \S\ref{1dint2o}, the Taylor expansion of $C_{2I}(\tilde s,\tau)$ at $\tilde s=0$ is
\begin{equation*}
C_{2I}(\tilde{s},\tau)=C_{2I}(0,\tau)+\tilde s C'_{2I}(0,\tau)+\mathcal{O}(\tilde{s}^2),
\end{equation*}
where
\begin{equation*}
\begin{split}
&C_{2I}(0,\tau)=\begin{bmatrix}
 -\frac{1}{2} & \frac{r}{2}  & -1+2\tau h & 0 & \frac{3}{2}+\frac{3r}{2}-2\tau h & -2r \\
-1 & 0 & 0 & 2 & -1 & 0 \\
1 & 0 & -\frac{1}{4} & -1 & \frac{1}{4} & 0 \\
\frac{1}{2r} & -\frac{1}{2} & \frac{3}{2}+\frac{3}{2r}-\frac{2\tau h}{r} & -\frac{2}{r} & -1+\frac{2\tau h}{r} & 0 \\
0 & -1 & -1 & 0 & 0 & 2 \\
0 & 1 & \frac{1}{4} & 0 & -\frac{1}{4} & -1
\end{bmatrix},  \\
&C'_{2I}(0,\tau)=\begin{bmatrix}
0 & 0 & 0 & 0 & 0 & 0 \\
0 & 0 & 0 & 0 & 0 & 0 \\
1 & 0 & 0 & 0 & 0 & 0 \\
0 & 0 & 0 & 0 & 0 & 0 \\
0 & 0 & 0 & 0 & 0 & 0 \\
0 & 1/r & 0 & 0 & 0 & 0 
\end{bmatrix}.
\end{split}
\end{equation*}

In \S\ref{1dint4o}, the Taylor expansion of $C_{4I}(\tilde s,\tau)$ at $\tilde s=0$ is 
\begin{equation*}
C_{4I}(\tilde{s},\tau)=C_{4I}(0,\tau)+\tilde s C'_{4I}(0,\tau)+\mathcal{O}(\tilde{s}^2),
\end{equation*}
where
\begin{equation*}
\begin{split}
&C_{4I}(0,\tau)=\begin{bmatrix}
-\frac{23}{17} & \frac{31-36\sqrt{3}}{17} & \frac{28r}{17} & \frac{4r(8\sqrt{3}-5)}{17} 
& \frac{48h\tau-34}{17} & \frac{13}{17} & \frac{44r-48h\tau+44}{17} & -\frac{72r}{17} \\
-1 & -1 & 0 & 0 & \frac{13}{59} & 2 & -\frac{72}{59} & 0 \\
\frac{55}{43} & \frac{12\sqrt{3}+85}{43} & 0 & 0 & -\frac{32}{43} & -\frac{59}{43} & \frac{36}{43} & 0 \\
-\frac{1}{49} & -\frac{8\sqrt{3}+37}{49} & 0 & 0 & \frac{9}{49} & 0 & -\frac{8}{49} & 0 \\
\frac{28}{17r} & \frac{4(8\sqrt{3}-5)}{17r} & -\frac{23}{17} & \frac{31-36\sqrt{3}}{17} 
& \frac{44r-48h\tau+44}{17r} & -\frac{72}{17r} & \frac{48h\tau-34r}{17r} & \frac{13}{17} \\
0 & 0 & -1 & -1 & -\frac{72}{59} & 0 & \frac{13}{59} & 2 \\
0 & 0 & \frac{55}{43} & \frac{12\sqrt{3}+85}{43} & \frac{36}{43} & 0 & -\frac{32}{43} & -\frac{59}{43} \\
0 & 0 & -\frac{1}{49} & -\frac{8\sqrt{3}+37}{49} & -\frac{8}{49} & 0 & \frac{9}{49} & 0
\end{bmatrix},\\
&C'_{4I}(0,\tau)=\begin{bmatrix}
-\frac{9}{17} & 0  & \frac{8}{17} & 0 & 0 & 0 & 0 & 0 \\
0 & 0 & 0 & 0 & 0 & 0 & 0 & 0 \\
\frac{51}{43} & 0 & 0 & 0 & 0 & 0 & 0 & 0 \\
-\frac{2}{49} & 0 & 0 & 0 & 0 & 0 & 0 & 0 \\
\frac{8}{17r} & 0  & -\frac{9}{17r} & 0 & 0 & 0 & 0 & 0 \\
0 & 0 & 0 & 0 & 0 & 0 & 0 & 0 \\
0 & 0 & \frac{51}{43r} & 0 & 0 & 0 & 0 & 0 \\
0 & 0  & -\frac{2}{49r} & 0 & 0 & 0 & 0 & 0
\end{bmatrix}.
\end{split}
\end{equation*}

\section*{Acknowledgement} 
This work has been partially financed by the grant 2014--6088 from the Swedish Research Council. The computations were performed on resources provided by the Swedish National Infrastructure for Computing (SNIC) at Uppsala Multidisciplinary Center for Advanced Computational Science (UPPMAX) under project snic2014-3-32.

\end{document}